\documentclass[12pt, reqno]{amsart}
	
\usepackage[utf8]{inputenc}
\usepackage[T1]{fontenc}
\usepackage{tikz,tikz-3dplot}
\usepackage{amsmath,amsfonts,amssymb,amsmath,latexsym,mathrsfs}
\usepackage{shuffle}
\usepackage[all]{xy}
\usepackage{stackengine}

\usepackage{tabu}
\usepackage{tikz-cd} 

\usepackage{comment}
\usepackage{url}
\usepackage{tikz}
\newcommand{\Q}{ {\mathbb Q} }
\newcommand{\R}{ {\mathbb R} }
\newcommand{\C}{ {\mathbb C} }
\newcommand{\Z}{ {\mathbb Z} }
\newcommand{\N}{ {\mathbb N} }

\newcommand{\PSU}{\mathrm{PSU}}

\newcommand{\PGL}{\mathrm{PGL}}

\newcommand{\PP}{ {\mathbb P} }

 \newcommand{\gr}{ {\mathrm{gr}} }
\newcommand{\Hom}{ \mathrm{Hom}}

\newtheorem{theorem}{Theorem}[section]

\newtheorem{proposition}[theorem]{Proposition}
\newtheorem{definition}[theorem]{Definition}
\newtheorem{corollary}[theorem]{Corollary}
\newtheorem{prop-def}[theorem]{Proposition-Definition}
\newtheorem{remark}[theorem]{Remark}

\newtheorem{lemma}[theorem]{Lemme}

\title{On models of orbit configuration spaces of surfaces}
\author{Mohamad Maassarani}
 \address{IRMA, Université de Strasbourg, 7 rue René Descartes, 67084
Strasbourg, France}
\email{maassarani@math.unistra.fr}

\begin{document}

\maketitle 
\begin{abstract}
We consider orbit configuration spaces $C_n^G(S)$, where $S$ is obtained by removing a finite number of points (eventually none) to a compact orientable boundaryless surface $\bar{S}$ and $G$ is a finite group acting freely continuously on $S$. We prove that the fibration $\pi_{n,k} : C_{n}^G(S) \to C_k^G(S)$ obtained by projecting on the first $k$ coordinates is a rational fibration. As a consequence, the space $C_{n}^G(S)$ has a Sullivan model $A_{n,k}=\Lambda V_{C_k^G(S)}\otimes \Lambda V_{C_{n-k}^G(S_{G,k})}$ fitting in a cdga sequence: $\Lambda V_{C_k^G(S)}\to A_{n,k} \to \Lambda V_{C_{n-k}^G(S_{G,k})},$ where $\Lambda V_X$ denotes the minimal model of $X$, and $C_{n-k}^G(S_{G,k})$ is the fiber of $\pi_{n,k}$. We show that this model is minimal except for some cases when $S\simeq S^2$ and compute in all the cases the higher $\psi$-homotopy groups (related to the generators of the minimal model in degree greater than $2$) of $C_n^G(S)$. We deduce from the computation that the space $C_n^G(S)$ having finite Betti numbers is a rational $K(\pi,1)$, i.e its minimal model and $1$-minimal model are the same (or equivalently the $\psi$-homotopy space vanishes in degree grater then $2$), if and only if $S$ is not homeomorphic to $S^2$. In particular, for $S$ not homeomorphic to $S^2$, the minimal model (isomorphic to the $1$-minimal model) is entirely determined by the Malcev Lie algebra of $\pi_1 C_n^G(S)$.  When $A_{n,k}$ is minimal, we get an exact sequence of Malcev Lie algebras $0\to L_{C_{n-k}^G(S_{G,k})}\to L_{C_{n}^G(S)}\to L_{C_k^G(S)}\to 0$, where $L_X$ is the Malcev Lie algebra of $\pi_1X$.  For $S \varsubsetneq \bar{S}=S^2$ and $G$ acting by orientation preserving homeomorphism, we prove that the cohomology ring of $C_n^G(S)$ is Koszul, and that for the spaces of this type studied in \cite{MM} the minimal model can be obtained out of a Cartan-Chevally-Eilenberg construction applied to a graded Lie algebra given in loc. cit..
\end{abstract}

\section*{Introduction}
\subsection*{Context and main results}
For $M$ a topological space, $H$ a group acting by homeomorphisms on $M$ and $n \geq 1$ an integer, the orbit configuration space of $n$ ordered points
of $M$ (associated to the action of $G$) is the space:
$$C_n^H(M)=\{(p_1,\dots,p_n)\in M^n \vert p_i\notin H \cdot p_j \}.$$
They are generalizations of the classical configuration space $C_n(M)$ of $n$ ordered points of $M$ (corresponding to the case $G=\{1\}$). Their topology and their algebraic invariants are widely studied.\\\\
A model $A$ of a topological space $X$ is a differential graded commutative algebra (cdga) quasi-isomorphic to the cdga $A_{PL}(X)$ of polynomial differential forms of $X$. A model is called minimal (Sullivan) model if the cdga is a free graded commutative algebra (i.e. of the form $\Lambda V_X$, with $V_X$ graded) and the differential is decomposable. The minimal model is unique up to "isomorphism" and in \cite{HarRF}, the author uses the notion $\psi$-homotopy space of $X$ to designate the graded vector space $V_X$. A weaker, invariant of $X$ is its $1$-minimal model: a free cdga $\Lambda V_{X,1}$ with $V_{X,1}$ concentrated in degree $1$ and a morphism $A\to A_{PL}(X)$ inducing in cohomology an isomorphism in degree $1$ and an monomorphism in degree $2$. If the first cohomology group of $X$ (or equivalently of $A_{PL(X)}$) is finite dimensional the Malcev Lie algebra $L_X$ of $\pi_1X$ can be constructed out of the $1$-minimal model of $X$ and conversely. A path connected topological space $X$ with finite Betti numbers is a rational $K(\pi,1)$ if its $1$-minimal model is isomorphic to its minimal model or equivalently the $\psi$-homotopy space of $X$ is null in degree $k\geq 2$ (Cf. Subsection \ref{S12} for more details). In particular, the minimal model is determined by the Malcev Lie algebra of the fundamental group.\\\\
Models for configuration spaces $C_n(M)$, with $M$ a smooth compact complex projective variety, were computed in \cite{FM} and then simplified in \cite{Kriz}. These models were used in \cite{Bezr}, where the author proves that the Malcev Lie algebra $L_{C_n(M)}$ of $\pi_1C_n(M)$, for $M$ a closed (compact and boundaryless) orientable surface, is the degree completion of a given graded Lie algebra $L(n,M)$, and gives the minimal Sullivan model of $C_n(M)$ in terms of a variant of the Cartan-Chevally-Eilenberg construction applied $L(n,M)$. In fact, configuration spaces of such surfaces are rational $K(\pi,1)$ spaces.\\\\
Earlier, it has been proved in \cite{FMmin}, fiber-type arrangements, that share some common topological properties with configuration spaces, are rational $K(\pi,1)$ spaces and models of such spaces are given. A general result, relating the Kozsulity of the cohomology ring of a formal spaces to the rational $K(\pi,1)$ property, is proved in \cite{PapYuz}. In \cite{BibbyHil}, it is shown that the complement $Y$ of arrangements associated to chordal graphs and closed orientable surfaces of positive genus are rational $K(\pi,1)$, and the minimal model and the Malcev Lie algebra of $Y$ are computed. The family of such spaces includes configuration spaces over the corresponding surfaces (those studied in \cite{Bezr}). In an operadic context, models for configuration spaces of framed points of surfaces are considered in \cite{IdCamWill}.\\\\
Here we study orbit configuration spaces $C_n^G(S)$, where $G$ is a finite group acting freely on a surface $S$ obtained by removing a finite number of points (eventually none) to a closed orientable surface $\bar{S}$. In \cite{DCoh} and \cite{BE}, the authors compute using different approaches the Malcev Lie algebra of $\pi_1C_n^{\mu_p}(\C^\times)$, where $\mu_p$ is the group of $p$ roots of the unity acting by multiplication on $\C^\times$. The cohomology ring of the orbit configuration space $C_n^{\Z/2\Z}(S^2)$ where the action is the antipodal action is studied in \cite{XicoAnt1} and \cite{ziegAnt} (see also \cite{XicoAnt2}). The Malcev Lie algebra of $\pi_1 C_n^{E_p}(E)$ where $E$ is an elliptic curve and $E_p$ is its group of $p$-torsion points acting naturally on $E$ is given in \cite{CalMart}. The orbit configuration spaces $C_n^G(S)$ for $\bar{S}=S^2$ and $G$ a group acting by orientation preserving homeomorphism are studied in \cite{MM} and \cite{MM3}. It is shown that they are formal spaces, the  cohomology ring and (for some cases) the Malcev Lie algebra are computed. One can also find in the literature (\cite{XicoCohGauss}, \cite{CKX}, \cite{casto}), orbit configuration spaces associated to infinite (geometric) group actions on $\C$ and the hyperbolic upper half plane.\\\\
\textbf{Main results}
By a general result the projection on the first $k$ coordinates $\pi_{n,k} : C_{n}^G(S) \to C_k^G(S)$ for $n>k$ is a locally trivial fibration. The fiber is the orbit configuration space $C_{n-k}^G(S_{G,k})$, where $S_{G,k}$ is $S$ with $k$ orbits removed. We show that: 
\begin{itemize}
\item[1)] The fibration $\pi_{n,k} : C_{n}^G(S) \to C_k^G(S)$ is a rational fibration. In particular, the space $C_{n}^G(S)$ has a Sullivan model $A_{n,k}=\Lambda V_{C_k^G(S)}\otimes \Lambda V_{C_{n-k}^G(S_{G,k})}$ fitting in a cdga sequence: $$\Lambda V_{C_k^G(S)}\to A_{n,k} \to \Lambda V_{C_{n-k}^G(S_{G,k})},$$ where $\Lambda V_X$ denotes the minimal model of $X$, the first morphism is $\mathrm{id}\otimes 1$ and the second is obtained by applying the augmentation map of $\Lambda V_{C_k^{G}(S)}$. 
\item[2)] The Sullivan model $A_{n,k}$ of $C_{n}^G(S)$ is minimal except for some cases when $S\simeq S^2$.
\item[3)] If $S$ is not homeomorphic to $S^2$, then $C_n^G(S)$ is a rational $K(\pi,1)$ space, i.e. the 1-minimal model and the minimal model of $C_n^G(S)$ are isomorphic, and hence they are determined by the Malcev Lie algebra $L_{C_n^G(S)}$ of $\pi_1 C_n^G(S)$.
\item[4)] If $S\simeq S^2$, then up to equivalence $C_n^G(S)$ is $C_n(S)$ or $C_n^{\Z/2\Z}(S^2)$, where $\Z/2\Z$ is the group generated by the antipodal map and the $\psi$-homtopy space in degree $k\geq 2$ is equal to $(\Q)_3$ if $G=\{1\}$ and $n\geq 3$ or $G\neq \{1\}$ and $n\geq 2$, and to  $(\Q)_2\oplus (\Q)_3$ otherwise (subscripts correspond to the degree). In particular, the orbit configuration spaces $C_n^G(S^2)$ are not rational $K(\pi,1)$ spaces.
\item[5)] When $A_{n,k}$ is minimal, there is an exact sequence of Malcev Lie algebras:
$$0\to L_{C_{n-k}^G(S_{G,k})}\to L_{C_{n}^G(S)}\to L_{C_k^G(S)}\to 0,$$
where $L_X$ is the Malcev Lie algebra of $\pi_1X$, and in all cases $L_{C_{n}^G(S)}$ is obtained form $L_S$ by iterated extensions by complete free Lie algebras.
\end{itemize}
For $\bar{S}=S^2$ and $G$ acting by orientation preserving homeomorphisms, we prove that:
\begin{itemize}
\item[6)] The minimal models of the spaces studied in \cite{MM} (except one case) can be obtained from a Cartan-Chevally-Eilenberg construction applied to graded Lie algebras from loc. cit..
\item[7)] If $G\neq\{1\}$ ($S\varsubsetneq S^2$, otherwise there is no such free actions), the cohomology ring of the orbit configuration space $C_n^G(S)$ is Koszul.
\end{itemize} 
The existence of the model for $A_{n,k}$ in $(1)$, is just an application of a general construction known for fibrations that are rational fibrations. Result $(2)$ for $S$ not homeomorphic to $S^2$ and $(3)$ are not hard to prove using an induction. For $S=S^2$, $(2)$ is obtained by combining the existence of cross-sections to a criterion from \cite{HarRF}. The $\psi$-homotopy groups in $(4)$ are obtained by induction using $(2)$ and some technical elements. The exact sequence in $(5)$ is derived from the sequence in $(1)$ using general arguments. In general, the $1$-minimal model of $X$, when $H^1(X,\Q)$ is finite dimensional, can be recovered as an inverse limit of Cartan-Chevally-Eilenberg cdga's associated to quotients of $L_X$. For the spaces in $(6)$, the Malcev Lie algebra is the degree completion of a graded Lie algebra and it is not hard to describe the inverse limit mentioned before. This is a general argument that also applies for other orbit configuration spaces (Cf. end of subsection \ref{S42}) and can be compared to a result of \cite{Bezr} for $C_n(S_g)$ where $S_g$ is a closed orientable surface of genus $g\geq 1$. Finaly result $(7)$, can be obtained by combining $(3)$ and a result from \cite{PapYuz}, or by applying an argument from \cite{Prid} to the cohomology ring computed in \cite{MM3}.
\subsection*{Outline of the paper}
The first section contains reminders on (1-)minimal models, $\psi$-homotopy groups, rational $K(\pi,1)$ spaces, rational fibrations, their models, and criterions for rationality of fibrations.\\\\
In section \ref{S2}, we recall the definition of an orbit configuration space and a fibration theorem. We prove main result $(1)$ and other results that will be used in section \ref{S3}. \\\\
In section \ref{S3}, we prove main result $(2)$, the part of main result $(3)$ not concerning Malcev Lie algebras and main result $(4)$. In section \ref{S4}, we recall some material on Malcev Lie algebras and their relation between minimal models ($1$-minimal models) and Malcev Lie algebras. This establishes the part of $(3)$ concerning Malcev Lie algebras. We then prove main results $(5)$ and $(6)$. Two proofs of main result $(7)$ are discussed section \ref{S5}.\\\\
One can construct Malcev completions using different techniques and it is known that they are all equivalent. We prove that the construction using complete Hopf algebras from \cite{Q1} and the one using fundamental groups of minimal models from \cite{FelHar2} are equivalent. We use this correspondence and general arguments to prove a result used in section \ref{S4}. 

\renewcommand{\abstractname}{\textbf{Acknowledgments}}
\begin{abstract}
The author warmly thanks the "Institut de Recherche Mathématique Avancée" in Strasbourg, for its hospitality.
\end{abstract}
\tableofcontents
\section{Reminders}\label{S1}
In the first $3$ subsection, we remind elements on: models, minimal models, rational $K(\pi,1)$ spaces, and models for fibrations and rational fibrations. Most of the material can be found in \cite{FelHar} and a reference is given when the material is taken elsewhere. In the last section, we recall a criterion for rationality of fibrations under a nilpotence assumption on the action of the fundamental group of the base on the cohomology of the fiber, and consider related actions. This is done to simplify the exposition of section \ref{S2}. 
\subsection{Models in the cdga category}
By a commutative cdga (over $\Q$) we mean a positively graded differential graded commutative $\Q$-algebra, i.e. a graded unital commutative $\Q$-algebra $A=\oplus_{i \geq 0}A^i$ (with the Koszul sign convention) equipped with a graded map $d$ of degree +1 (the differential) satisfying the derivation rule:
$$d(ab)=d(a)b+(-1)^{\mathrm{deg}(a)\mathrm{deg}(b)}a d(b),$$ for $a$ and $b$ homogeneous elements of degree $\mathrm{deg}(a)$ and $\mathrm{deg}(b)$ respectively. We will designate a cdga by the couple $(A,d)$ if necessary or simply by $A$ if the context is clear. The cohomolgy algebra of $A$ will be denoted by $H^*(A)$, and $A$ is called connected if $H^0(A)=\Q$. A cdga morphism $f:A\to B$ between two cdga's $A$ and $B$ is a morphism of graded algebra respecting the differentials. Its associated morphism in cohomology will be denoted by $H^*(f)$. The cdga morphism $f$ is called quasi-isomorphism if $H^*(f)$ is an isomorphism. An augmented cdga is a cdga $A$ with a cdga morphism $\varepsilon_A: A\to \Q$ (the augmentation map) where $\Q$ is concentrated in degree $0$. A morphism of augmented cdga is a cdga morphism respecting the augmentations. \\\\
Let $V=\oplus_{i\in \N^*} V^i$ be a graded vector space. The free commutative algebra $\Lambda V$ on $V$ is the quotient of the graded tensor algebra $T(V)$ by the ideal generated by the elements $ab+(-1)^{\mathrm{deg}(a)\mathrm{deg}(b)}ba$, for $a$ and $b$ homogeneous elements of $V$, with grading induced by the one over $V$. The natural augmentation on $T(V)$ induces an augmentation on $\Lambda V$. We denote by $\Lambda^+ V$ the kernel of the augmentation map. For $V'$ strictly positively graded $\Lambda (V\oplus V')$ is isomorphic to the tensor product $\Lambda V \oplus \Lambda V'$ (tensor with the Koszul sign convention). \\\\
A Sullivan algebra is a cdga $(\Lambda V ,d)$ ($V=\oplus_{i\in \N^*} V_i$) with $V = \cup_{k=0}^\infty V(k)$, where $V(0) \subset V(1) \subset \cdots$ is an increasing sequence of graded subspaces such that
$dV(0)=0$ and $dV(k) \subset \Lambda V(k - 1)$, for $k> 1$. Such algebra is called minimal if $dV \subset \Lambda^+ V \cdot\Lambda^+V$.\\\\
A relative Sullivan algebra is a cdga $B\otimes \Lambda V$ where $B\otimes 1$ is a connected sub-cdga, $V=\oplus_{i\in \N^*}$, and $V = \cup_{k=0}^\infty V(k)$, where $V(0) \subset V(1) \subset \cdots$ is an increasing sequence of graded subspaces such that
$dV(0)\subset B$ and $dV(k) \subset B\otimes \Lambda V(k - 1)$, for $k\geq 1$. Such a cdga is called minimal if $\mathrm{Im}(d) \subset B^+ \otimes \Lambda V + B \otimes (\Lambda^+ V \cdot\Lambda^+V)$, where $B^+=\oplus_{i\geq 1} B^i$. We will identefy $B$ with $B\otimes 1$ and $1\otimes \Lambda V$ with $\Lambda V$ throughout the rest of the text. Sullivan algebras defined above are relative Sullivan algebras with $B=\Q$.\\\\
An augmentation $\varepsilon : B \to \Q$ gives an augmentation for the relative Sullivan algebra $B\otimes \Lambda V$, since $\Lambda V$ is naturally augmented. Moreover, composing the restriction of $d$ to $\Lambda V$ with $\varepsilon \otimes \mathrm{id}$ we obtain a differential $\bar{d}$ for which $\Lambda V$ is a Sullivan algebra, and $\varepsilon\otimes \mathrm{id}$ gives and cdga morphism $\varepsilon\cdot 1 :(B\otimes \Lambda V,d) \to (\Lambda V,\bar{d})$. If $B\otimes V$ is minimal then so is $(\Lambda V, \bar{d})$. The morphism $\varepsilon \cdot 1$ will always be regarded as a cdga morphism with respect to the differential $d$ and $\bar{d}$.
\begin{definition}
\begin{itemize}
\item[1)] A Sullivan model of a morphism $\varphi: B \to C$ of connected cdga's is a pair $(B\otimes \Lambda V, m_B)$, where $B\otimes \Lambda V$ is a relative Sullivan algebra and $m_B: B \otimes \Lambda V \to C$ is a quasi-isomorphism extending $\varphi$. The Sullivan model is called minimal if $B\otimes \Lambda V$ is minimal.
\item[2)] A Sullivan model of a connected cdga $A$ is a pair $(\Lambda V,m)$, where $\Lambda V$ is a Sullivan algebra and $m: \Lambda V \to A$ is a quasi-isomorphism. We say that $(\Lambda V,m)$ is minimal if $\Lambda V$ is minimal.
\end{itemize}
\end{definition}
A minimal model of a cdga is a minimal model of the inclusion of $\Q$ into the corresponding cdga. Note that if $A$ is augmented then $m$ is a morphism of augmented cdga's. The same is true for $m_B$ if $\varphi$ is a map of augmented cdga's. Sullivan models (resp. minimal models) defined here (following \cite{FelHar}) are also models (resp. minimal models) in the sense of \cite{HarRF} (p. 202-203) where the author uses a kind of well-ordered basis for the definition of a model and a "degree condition" to define the notion of minimal model. One derives from the filtration here the well-ordered basis. The "degree condition" is equivalent to the "minimality" condition here as "noted" in \cite{HarRF} when the definition of minimal KS complex is introduced (p. 202).
\begin{remark}
We will somtimes designate a minimal model $(\Lambda V ,m)$ simply by $\Lambda V$.
\end{remark}
\begin{proposition}
Take $A,B,C$ and $\varphi$ as in the previous definition.
\begin{itemize}
\item[1)] If $H^1(\varphi)$ is a monomorphism then $\varphi$ admits a minimal model. Moreover, if $(B\otimes \Lambda V,m_B)$ and $(B \otimes \Lambda V', m'_B)$ are such models, then we have a cdga isomorphism $B\otimes \Lambda V \to B\otimes \Lambda V'$ extending $\mathrm{id}_B$.
\item[2)] Every connected cdga $A$ admits a minimal model. Moreover, If $(\Lambda V,m)$ and $(\Lambda V',m')$ are such models, then $ \Lambda V$ and $\Lambda V'$ are isomorphic Sullivan algebras.
\end{itemize}
\end{proposition}
We note that one can relate $m$ to $m'$ by a homotopy. We omit this in the text.
\begin{remark}
If $V$ is concentrated in degree $1$ then any Sullivan algebra $(\Lambda V,d)$ is minimal.
\end{remark}
%In fact the maps $m$ and $\varphi m'$ are homotopic. But we will not use the notion of homotopy.
\begin{definition}
A $1$-minimal model of a connected cdga $A$ is a Sullivan algebra $(\Lambda V_1, d)$ where $V_1$ is of degree $1$ and a cdga morphism $m_1: \Lambda V_1 \to A$ such that
$H^1(m_1)$ is an isomorphism and $H^2(m_1)$ is a monomorphism.
\end{definition}
The definition of $1$-minimal model given here is equivalent to the one in \cite{MGRH} using Hirsch extensions.
\begin{proposition}[\cite{MGRH}]
Every connected cdga $A$ admits a $1$-minimal model. Moreover, If $\Lambda V_1,m_1$ and $\Lambda V_1',m_1'$ are such models, then the Sullivan algebras $ \Lambda V_1 $ and $\Lambda V_1'$ are isomorphic.
\end{proposition}
\begin{remark}\label{rmk min 1min}
If $(\Lambda V,d)$ is a minimal Sullivan algebra then $d$ restricts to  $\Lambda V_1$ and the natural inclusion $(\Lambda V_1,d)\to (\Lambda V,d)$ provides a $1$-minimal model of $(\Lambda V,d)$. This can be shown by direct comparaison between first and second cohomology groups of both Sullivan algebras. In particular, one recovers the $1$-minimal model of a cdga from the minimal model.
\end{remark}
\subsection{Models for topological spaces, rational $K(\pi,1)$ spaces and formality}\label{S12}
One has a contravariant functor $A_{PL}$ from the category of topological spaces into the category of cdga's over $\Q$. For $X$ a topological space, $A_{PL}(X)$ is called the cdga of polynomial differential forms on $X$ with values in $\Q$. A Sullivan model (resp. minimal model) of $X$ is a Sullivan model (resp. minimal model) for $A_{PL}(X)$. The algebra $A_{PL}(*)$ for $*$ a point is isomorophic to $\Q$. Hence, $A_{PL}$ induces a contravariant functor from pointed topological spaces to augmented cdga's with augmentations given by $A_{PL}(i(*))$ where $i(*)$ is the inclusion of the base point.\\\\
Let $A$ be a Sullivan algebra. Denote by $\varepsilon$ the augmentation of $A$ and set $Q(A)= \mathrm{Ker}(\varepsilon)/ (\mathrm{Ker}(\varepsilon)\cdot \mathrm{Ker}(\varepsilon))$ ($Q(\Lambda V)=V$). The differential $d$ of $A$ induces a differential $Q(d)$ on $Q(A)$. We denote by $H^*(Q(A))$ the cohomology of $Q(A)$ with respect to $Q(d)$ ($H^0(Q(A))=0$). As one can check the Sullivan algebra $A$ is minimal if and only if $Q(d)=0$ (equivalently $H^*(Q(A))=Q(A)$). If $f : A \to B$ is a morphism between Sullivan algebras then $f$ induces naturally a morphism $f^\#:H^*(Q(A))\to H^*(Q(B))$ of graded spaces.
\begin{proposition}\label{quasi diez}
If $f:A \to B$ is a cdga morphism between Sullivan algebras then $f$ is a quasi-isomorphism if and only $f^\#$ is an isomorphism.
\end{proposition}
\begin{comment}
Page 153, lifting property
\end{comment}
Two Sullivan models of a given cdga are quasi-isomorphic. Hence, by the previous proposition $H^*(Q(A))\simeq H^*(Q(A'))$ for $(A,m)$, $(A',m')$ two different Sullivan models of $B$. The $\psi$-homotpy space $\pi_\psi^*(B)$ of $B$ (\cite{HarRF}) correspond to the isomorphism class of the vector spaces $H^*(Q(A))$ for $A$ running over Sullivan models of $B$. When the context is not ambiguous, we will designate $\pi_\psi^*(B)$ by a representative of the class and write $\pi_\psi^*(B)=V$ (for a given representative) or say $\pi_\psi^*(B)$ is isomorphic to $V$. For instance, if $\Lambda V$ is a minimal model of $B$ then $\pi_\psi^*(B)=V$. The $\psi$-homotopy space $\pi_\psi^*(X)$ of a topological space $X$ is the $\psi$-homotopy space of $A_{PL}(X)$. We denote by $\pi_\psi^k(X)$ the degree $k$ component of $\pi_\psi^*(X)$ and call it the $k$-th $\psi$-homotopy group of $X$.\\\\
We have assigned to a topological space $X$ a cdga of polynomial forms $A_{Pl}(X)$. In fact, one assigns to a simplicial set $X_s$ a cdga of polynomial forms $A_{PL}(X_s)$ (\cite{FelHar}, \cite{BousGug}) and $A_{PL}(X)$ is equal to $A_{PL}(S_*(X))$ where $S_*(X)$ is the simplicial set of singular complexes on $X$. The Bousfield-Kan $\Q$-completion (\cite{BousKan}) assigns to a simplicial set $X_s$ a simplicial set $\Q_{\infty} X_s$. We recall that the homotopy groups of a simplicial set $X_s$ correspond to those of its realization $\vert X_s \vert$. If $X_s$ has finite dimensional rational cohomology, then $\pi_k\Q_{\infty} X_s \simeq \pi_\psi^k(A_{PL}(X_s))$, for $k\geq 2$ (\cite{BousGug}, theorem 12.8). Hence, for $X$ a topological space with finite Betti numbers, $\vert \Q_{\infty}S_*(X) \vert$ is a $K(\pi,1)$ if and only if $\pi_\psi^k(X)=0$ for $k\geq 2$. The latter condition is equivalent to the following (Cf. remark \ref{rmk min 1min}): the $1$-minimal model of $X$ is the minimal model of $X$. The notion of rational $\Q(\pi,1)$ appears in different papers in the literature (\cite{FMmin}, \cite{PapYuz}, \cite{PapSucChenLie}):
\begin{definition}
A topological space $X$ with finite betti numbers is a rational $K(\pi,1)$ if one of the following equivalent conditions is satisfied:
\begin{itemize}
\item[1)] The topological space $\vert \Q_{\infty}S_*(X) \vert$ is a $K(\pi,1)$.
\item[2)] The minimal model and the $1$-minimal model of $X$ are isomorphic.
\item[3)] For $k\geq 2$, $\pi_\psi^k(X)=0$.
\end{itemize}
\end{definition}
One can wonder whether a $K(\pi,1)$ space is a rational $K(\pi,1)$ space. The answer is no. A counter example is given in \cite{FMmin}.
\begin{remark}
If $X$ is a CW complex with finite dimensional first rational homology group, then the Malcev algebra of $\pi_1X$ determines the $1$-minimal model of $X$ (Cf. section \ref{S4}). Hence, if moreover $X$ is a rational $K(\pi,1)$, then the Malcev Lie algebra of $X$ determines the minimal model of $X$.
\end{remark}
\begin{theorem}[\cite{PapYuz}]\label{Koszul KP1}
If $X$ is a formal topological space with finite Betti numbers, then $X$ is a rational $K(\pi,1)$ if and only if the graded cohomology algebra $H^*(X,\Q)$ is Koszul.
\end{theorem}
The notion of formal space can be found in (\cite{FelHar}, P.156) and the notion of Koszul algebra (for graded algebras) is largely covered in \cite{BeilKoszul}. Koszul algebras are also defined in \cite{Prid} (where the definition coincide with the latter in the graded case). The notions themselves are not essential for the work.  
\subsection{Rational fibrations and models}\label{S13}
For $X$ a topological space, one has a natural isomorphism between the singular cohomology $H^*(X,\Q)$ and the cohomology of $A_{PL}(X)$. Hence, for $f$ a map between topological spaces, $H^*(f,\Q)$ can be identified to $H^*(A_{PL}(f))$.\\\\
In this subsection, $F\overset{J}{\to} E \overset{\pi}{\to} B$ is a fibration where $F,E,B$ are pointed path connected spaces and $J,\pi$ are maps of pointed spaces. The map $\pi$ induces a surjective map on fundamental groups (this follows from the homotopy long exact sequence of the fibration) and therefore a surjective map on first singular homology groups. As a consequence, $H^1(\pi,\Q)$ and $H^1(A_{PL}(\pi))$ are injective and $A_{PL}(\pi)$ admits a minimal relative Sullivan model $(A_{PL}(B)\otimes \Lambda Z, m)$. Denote by $\varepsilon_B$ the augmentation of $A_{PL}(B)$. The map $A_{PL}(J)m $ factors throught $\varepsilon_B\cdot \mathrm{id}: A_{PL}(B) \otimes \Lambda Z\to \Lambda Z$ to give the commutative diagram of cdga morphisms:
\[\begin{tikzcd}\label{A}
A_{PL}(B) \arrow[d,equals] \arrow{r}{i_B}& A_{PL}(B) \otimes \Lambda Z \arrow{d}{m} \arrow{r}{\varepsilon_B\cdot \mathrm{id}} & \Lambda Z \arrow{d}{m''}\\
A_{PL}(B) \arrow{r}[swap]{A_{PL}(\pi)} & A_{PL}(E) \arrow{r}[swap]{A_{PL}(J)} & A_{PL}(F)
\end{tikzcd},\qquad (A)\]
where $i_B(x)=x\otimes 1$.
\begin{definition}[\cite{HarRF}]\label{def RF}
The fibration $F\overset{J}{\to} E \overset{\pi}{\to} B$ is called a rational fibration, if $m''$ is a quasi-isomorphism.
\end{definition}
Thus the fibration is a rational fibration if and only if $(\Lambda Z, m'')$ is the minimal model of $A_{PL}(F)$, since $\Lambda Z$ is minimal. Generally, a rational fibration (\cite{HarRF}) is a sequence of maps of pointed topological spaces $X_1 \to X_2 \to X_3$ satisfying a similar condition. Here, we only consider the case of fibrations.\\\\
Denote by $\varepsilon_B$ and $(\Lambda V, m_B)$ the augmentation map and the minimal model of $A_{PL}(B)$ respectively. Let $\varepsilon_V$ be the augmentation map of $\Lambda V$, and $(\Lambda V \otimes \Lambda W,m')$ a minimal relative Sullivan model $A_{PL}(\pi)i_B$. As for $(A)$, we obtain the following commutative diagram of cdga morphisms:
\[\begin{tikzcd}
\Lambda V \arrow{d}{m_B} \arrow{r}{i_V}& \Lambda V \otimes \Lambda W \arrow{d}{m'} \arrow{r}{\varepsilon_V\cdot \mathrm{id}} & \Lambda W \arrow{d}{m_Z}\\
A_{PL}(B) \arrow{r}[swap]{i_B}& A_{PL}(B) \otimes \Lambda Z \arrow{r}[swap]{\varepsilon_B\cdot \mathrm{id}} & \Lambda Z
\end{tikzcd},\qquad (B)\]
where $i_V(x)=x\otimes 1$, $m_B,m'$ are quasi-isomorphisms and $\Lambda V,\Lambda W$ and $\Lambda Z$ are minimal Sullivan algebras. The fact that $m'$ and $m_B$ are quasi-isomorphisms imply that $m_Z$ is a quasi-isomorphism (\cite{HarLMM}), theorem 7.1 p. 96). Hence, $(\Lambda W, m_Z)$ is a minimal model for $\Lambda Z$ which is itself a minimal Sullivan algebra. Therefore, $m_Z$ is an isomorphism by uniqueness of the minimal model.\\\\
Combining diagrams $(A)$ and $(B)$ and using the fact that $m_Z$ is an isomorphism, we get the commutative diagram of cdga morphisms:
\[\begin{tikzcd}
\Lambda V \arrow{d}{m_B} \arrow{r}{i_V}& \Lambda V \otimes \Lambda W \arrow{d}{m_E} \arrow{r}{\varepsilon_V\cdot \mathrm{id}} & \Lambda W \arrow{d}{m_F}\\
A_{PL}(B) \arrow{r}[swap]{A_{PL}(\pi)}& A_{PL}(E) \arrow{r}[swap]{A_{PL}(J)} & A_{PL}(F)
\end{tikzcd},\qquad (C)\]
where $i_V(x)=x\otimes 1$, $\Lambda V$ and $\Lambda W$ are minimal Sullivan algebras, $\Lambda V \otimes \Lambda W$ is a minimal relative Sullivan algebra, and $m_B$ and $m_E$ are a quasi-isomorphisms. Moreover, the fibration $F\to E\to B$ is a rational fibration if and only if $(\Lambda W, m_F)$ is a minimal model for $A_{PL}(F)$.
\begin{proposition}\label{prop model RF}
The relative Sullivan algebra $\Lambda V \otimes \Lambda W \simeq \Lambda (V\oplus W)$ is a Sullivan algebra, i.e. we can equip $\Lambda (V\oplus W)$ with a filtration satisfying the conditions in the definition of a Sullivan algebra. In particular, $(\Lambda V \otimes \Lambda W,m_E)$ is a Sullivan model for $E$.
\end{proposition}
\begin{proof}
One can use the argument in $(ii)$ of the proof of proposition 15.5 (p. 198) of \cite{FelHar}.
\end{proof}
\begin{comment}
Denote by $d$ the differential of $\Lambda V \otimes \Lambda W\simeq \Lambda (V\oplus W) $.
Set $S(-1)=0$ and define by induction $S(k)=\{ s \in V\oplus W \vert ds \in \Lambda S(k-1)\}$ for $k\geq 1$. It suffices to show that $V\oplus W=\cup_{k \geq 0} S(k)$. Since $\Lambda V$ is a Sullivan algebra, we have $d\Lambda V(k) \subset \Lambda V(k-1) $ (with $\Lambda V(-1):=0$). Hence, an induction proves that $V(k)\subset S(k)$ and therefore $V\subset \cup_{k \geq 0} S(k)$. Since $\Lambda V \otimes \Lambda W$ is a relative Sullivan algebra, $dW(k) \subset \Lambda( V \oplus W(k-1))$ (with $\Lambda W(-1):=1$). This and the fact that $V\subset \cup_{k\geq 0} S(k)$, proves that $ w\in W(k)$ lies in $S(k_w)$ for a given $k_w$ and therefore $W\subset \cup_{k\geq 0} S(k)$. We have proved that $\Lambda V \otimes \Lambda W$ is a Sullivan algebra. The fact that the couple in the proposition is a model for $E$ follows, since $m_E$ is a quasi-isomorphism.
\end{comment}
We will give a criterion from \cite{HarRF} allowing determining whether the Sullivan model $(\Lambda V \otimes \Lambda W,m_E)$ of $E$ is minimal or not. The top morphisms in diagram $(C)$ give an exact sequence of complexes:
$$0\to Q(\Lambda V) \overset{Q(i_V)}{\to} Q(\Lambda V \otimes \Lambda W) \overset{Q(\varepsilon_V\cdot 1)}{\to} Q(\Lambda W) \to 0,$$
and hence by the Zig-Zag lemma we have a long exact sequence:
\begin{align}\label{exact sequence}
\begin{split}
\cdots \to H^{k-1}(Q(\Lambda W)) &\to H^k(Q(\Lambda V))\overset{i_v^\#}{\to} H^k(Q(\Lambda V \otimes \Lambda W)) \\
&\overset{(\varepsilon_v\cdot 1)^\#}{\to} H^k(Q(\Lambda W)) \to H^{k+1}Q(\Lambda V) \to \cdots
\end{split}
\end{align}
We denote by $\partial^*$ the corresponding connecting homomorphism $H^*(Q(\Lambda W))\to H^{*+1}(Q(\Lambda V))$. Note that the groups in the long exact sequence correspond to the $\psi$-homotopy groups of $F,E$ and $B$. Since the long sequence is exact the following three conditions are equivalent:
$$(C1) \: \: \partial^*=0,\quad
(C2)\:\: \text{$i_v^\#$ is injective}, \quad
(C3)\:\: \text{$(\varepsilon_v\cdot 1)^\#$ is surjective},$$
where the maps are regarded as maps of graded spaces.
\begin{theorem}[\cite{HarRF}]\label{th minimal RF}
Assume that the fibration $F\to E \to B$ is a rational fibration. The Sullivan model $(\Lambda V \otimes \Lambda W,m_E)$ of $E$ is minimal if and only if one of the equivalent conditions $(C1),(C2), (C3)$ preceding the theorem is satisfied.
\end{theorem}
\begin{proof}
We know that the couple is a model for $E$. We show that the Sullivan algebra $\Lambda V \otimes \Lambda W$ is minimal. Denote by $d$ the differential of $\Lambda V \otimes \Lambda W$. Since $\Lambda V$ and $\Lambda W$ are minimal algebras, the induced differentials on $Q(\Lambda V)$ and $Q(\Lambda W)$ are trivial. It follows that $(C2)\Rightarrow Q(d)=0 \Rightarrow (C3)$. Hence, $Q(d)=0$ is equivalent to any of the conditions $(C1),(C2),(C3)$. This proves the result, since a Sullivan algebra $(A,d_A)$ is minimal if and only if $Q(d_A)=0$.
\end{proof}
\begin{corollary}\label{sec RF}
If the fibration $F\to E\overset{\pi}\to B$ is a rational fibration admitting a cross-section, then the Sullivan model $(\Lambda V \otimes \Lambda W,m_E)$ for $E$ is minimal.
\end{corollary}
\begin{proof}
Let $\Lambda V$ and $\Lambda V \otimes \Lambda W$ be the Sullivan algebras of diagram $(C)$. We have seen that they are models of $A_{PL}(B)$ and $A_{PL}(E)$. Assume that a cross-section $s$ exists. The map $A_{PL}(s)$, will lift to the models giving a diagram commuting up to a cdga homotopy (\cite{FelHar}, proposition 12.9, p. 153 ):
\[\begin{tikzcd}
\Lambda V \otimes \Lambda W \arrow{dr}{\tilde{s}}&\\
\Lambda V\arrow{r}{id} \arrow{u}{i_v} & \Lambda V
\end{tikzcd}\]
Hence, $H(\tilde{s}\circ i_v)=\mathrm{id}_{H(\Lambda V)}$ ( \cite{FelHar}, (1) of proposition 12.8, p. 152) and by proposition \ref{quasi diez}, $H^*(Q(\tilde{s}\circ i_v))= H^*(Q(\tilde{s})) \circ i_v^\#$ is an isomorphism. In particular, $i_v^\#$ is injective and the corollary follows from the previous theorem.
\end{proof}
If the fibration considered is a rational fibration, then the sequence (\ref{exact sequence}) can be rewritten as (up to fixing representatives):
\begin{align}\label{ex seq 2}
\cdots \to \pi_\psi^{k-1}(F) &\overset{\partial^*}\to \pi_\psi^k(B)\overset{i_v^\#}{\to} \pi_\psi^k(E) \overset{(\varepsilon_v\cdot 1)^\#}{\to} \pi_\psi^k(F) \overset{\partial^*}{\to} \pi_\psi^{k+1}(B) \to \cdots,
\end{align}
and the Sullivan model in theorem \ref{th minimal RF} is minimal if the connecting homomorphism $\partial^*$ here is trivial.

\subsection{Some rationality criterions for fibrations}\label{S14}
Let $U$ be an abelian group or a vector space and $G$ a group acting by automorphisms on $U$ (linear automorphisms if $U$ is a vector space). Set $U_0=0$ and 
$$U_n(G)=\underset{(g_1,\dots,g_n)\in G^n}{\cap}\mathrm{Ker}((g_1-\mathrm{id})\cdots (g_n-\mathrm{id})),$$ for $n\geq 1$. We say that $G$ acts nilpotently on $U$ if $U=U_k(G)$ for a given $k$ (\cite{HarLMM}). Note that this condition is equivalent to the existence of a sequence of subgroups (vector subspaces if $U$ is a vector space): 
$$0=U(0)\subset U(1) \subset \cdots \subset U(l)=U,$$ 
stable under $G$ and such that $G$ acts trivially on $U(n+1)/U(n)$ for $n \in[0,l-1]$. Indeed, the spaces $U_n(G)$ give such a sequence if the action in nilpotent. Conversely, if the sequence $U(i)$ exists, we have $(g-\mathrm{id})(U(n))\subset U(n-1)$ for $n\in [1,l], g\in G$,  and hence $U=U_l(G)$. 
\begin{remark}
If $G$ acts on $U$ nilpotently and $\phi: H \to G$ is a group morphism, then the action of $H$ on $U$ induced by $\phi$ is nilpotent. Indeed, $U_n(G)\subset U_n(H)$.
\end{remark}
The following theorem is theorem 4.6 of \cite{HarRF} (originally theorem 20.3 of \cite{HarRF}) where the statement is for Serre fibrations.
We fix $F\to E \to B$ be a fibration with $F,E,B$ path connected, chose $b\in B$ and $e$ an element in fiber $F_b$ over $b$.
\begin{theorem}[\cite{HarRF}, \cite{HarLMM}]\label{Th RF}
If either $H^*(F,\Q)$ or $H^*(B,\Q)$ is finite dimensional and $\pi_1(B,b)$ acts nilpotently on each $H^i(F,\Q)$ for $i\geq 1$, then $F\to E \to B$ is a rational fibration. 
\end{theorem}
The following proposition is proven in \cite{BousKan} (lemma 5.5, p. 64) for simplicial sets. The proof holds for topological spaces; replace $\pi_1(E_2)$ by $\pi_1(E_0)$ in the proof (it is probably a typo).
\begin{proposition}[\cite{BousKan}]\label{prop composition}
Let $F_1 \to  E_1 \overset{p}{\to} E_0$ and $F_2 \to E_2 \overset{q}{\to} E_1$ be fibrations with $F_1$ and $F_2$  path connected. If for $p$ and $q$ the fundamental group of the base acts nilpotently in all the rational homology groups of the fiber, then the fundamental group of $E_0$ acts nilpotently in all the rational homology groups of the fiber of $p\circ q$. 
\end{proposition}
\begin{corollary}\label{cor composition}
If $H_*(E_0,\Q)$ is finite dimensional then the fibration $p\circ q : E_2 \to E_0$ is a rational fibration.
\end{corollary}
Denote by $i_\#: \pi_1(F_b,e) \to \pi_1(E,e)$ the morphism induced by the natural inclusion $F_b\to E$. We assume that $i_\#$ is injective. Under this assumption, the homotopy long exact sequence of the fibration gives a short exact sequence:
 $$ 1 \to \pi_1(F_b,e)\overset{i_\#}{\to}  \pi_1(E,e) \overset{\pi_\#}{\to} \pi_1(B,b)\to 1,$$
where $\pi_\#$ is the morphism induced by the projection $E\to B$. We will regard $\pi_1(F_b,e)$ as a subgroup of $\pi_1(E,e)$ (using $i_\#$). Consider a map (of sets) $s: \pi_1(B,b) \to \pi_1(E,e)$, such that $\pi_\# s=\mathrm{id}$. We associate to $x\in \pi_1(B,b)$ the automorphism $\alpha_s(x): \pi_1(F_b,e)\to \pi_1(F_b,e), y\mapsto s(x)ys(x)^{-1}$. As one can check, the following defines an action by automorphisms of $\pi_1(B,b)$ on $\pi_1(F_b,e)^{ab}$ the abelianisation of $\pi_1(F_b,e)$:
\begin{align}\label{eq act} x\cdot\bar{z}=\overline{\alpha_s(x)(z)} \quad \text{for $(x,z)\in \pi_1(B,b)\times \pi_1(F_b,e)$}, 
\end{align} 
where $\bar{w}$ denotes the class of $w\in \pi_1(F_b,e)$ in $\pi_1(F_b,e)^{ab}$. Moreover, this action is independent of the choice of $s$.\\\\ 
We also have other actions related to the fibration, the action of $\pi_1(B,b)$ on the homology and cohomology of the fiber $F_b$. We recall the construction. Let $\gamma:[0,1] \to B$ be a loop with base point $b$. The homotopy $F_b\times [0,1] \to B,(w,t)\mapsto \gamma(t)$, lifts to a homotopy $H_\gamma: F_b\times[0,1] \to E$ such that: $H_\gamma(x,t) \in F_\gamma(t)$, $H_\gamma(-,0)$ is the natural inclusion $F_b \to E$, and $f_\gamma: H_\gamma(-,1)$ regarded as a self map of $F_b$ is a homotopy equivalence. The action of $\pi_1(B,b)$ in the homology and cohomology of the fiber $F_b$ (with coefficient in $R$) is induced by:
$$\gamma \cdot a=(f_\gamma)_*(a) , \quad  a'\cdot \gamma=(f_\gamma)^*(a')$$
for $\gamma$ a loop in $B$ with base point $b$,  $a\in H_*(F_b,R)$ and $a'\in H^*(F_b,R)$, and where $(f_\gamma)_*$ and $(f_\gamma)^*$ are the isomorphism induced by $f_\gamma$ regarded as self map of $F_b$ (with coefficients in $R$).
\begin{proposition}
If $i_\#: \pi_1(F_b,e) \to \pi_1(E,e)$ injective, then the action of $\pi_1(B,b)$ on $\pi_1(F_b,e)^{ab}$ described above is equivalent via the Hurewitz isomorophism to the action of $\pi_1(B,b)$ on the the first singular homology group $H_1(F_b,\Z)$ of the fiber. 
\end{proposition}
\begin{proof}
Let $\gamma:[0,1]\to B$ be a loop with base point $b$ and $H_\gamma:F_b \times [0,1]\to E$ the homotopy as described above. Take $\beta:[0,1] \to F_b$ a loop with base point $e$ and denote by $\beta^\gamma$ the loop $\beta^\gamma(t)=H_\gamma(\beta(t),1)$ based at $e_\gamma=H_\gamma(e,1)$. By definition, the action of $\pi_1(B,b)$ on $H^1(F_b,\Z)$ is induced by:
$$ \gamma \cdot [\beta]=[\beta^\gamma].$$
Note that $t\mapsto H_\gamma(e,t)$ defines a path $\tilde{\gamma}$ in $E$ from $e$ to $e_\gamma$ lifting $\gamma$. Now define $H_{\gamma,\beta}:[0,1]\times [0,1] \to E$ by $(t,s)\mapsto H_\gamma(\beta(t),s)$. The restrictions of $H_{\gamma,\beta}$ to $[0,1]\times 0$ and $\{1\}\times [0,1] \cup \{0\}\times [0,1]\times \{0\}$ are parametrisations of $\beta^\gamma$ and $\tilde{\gamma}\beta \tilde{\gamma}^{-1}$ and hence both loops define the same homotopy class in $\pi_1(E,e_\gamma)$. Therefore, if $\eta:[0,1] \to F_b$ is a path from $e_\gamma$ to $e$, then:
\begin{equation}\label{eq1111}
\overline{\eta\beta^\gamma\eta^{-1}}=\overline{\eta\tilde{\gamma}\beta \tilde{\gamma}^{-1}\eta^{-1}} \quad \text{ in $\pi_1(E,e)$}.\end{equation} 
Fix such an $\eta$. Since $\eta$ is a path in $F_b$ and $ \gamma \cdot [\beta]=[\beta^\gamma]$,
\begin{equation}\label{eq2222}
\gamma \cdot [\beta]= [ \eta\beta^\gamma\eta^{-1}]  \quad \text{ in $H_1(F_b,\Z)$}.\end{equation}
Notice that $\eta\tilde{\gamma}$ is a loop in $E$ based at $e$ lifing the loop $\gamma$ and hence its homotopy class in $\pi_1(E,e)$ is equal to $s(\gamma)$ for a given map of sets $s:\pi_1(B,b) \to \pi_1(E,e)$ satisfiying $\pi_\#s=\mathrm{id}$ (we denote the homotopy class of $\gamma$ in $\pi_1(B,b)$ aslo by $\gamma$). Therefore, by (\ref{eq1111}).
$$ \overline{\eta\beta^\gamma\eta^{-1}}=s(\gamma)\bar{\beta} s(\gamma)^{-1} \quad \text{in $\pi_1(E,e)$},$$
Hence, the image of the class of $\beta$ in $\pi_1(F_b,e)^{ab}$ under the action of $\gamma$ on $\pi_1(F_b,e)^{ab}$ is the class of the loop $\overline{\eta\beta^\gamma\eta^{-1}}$ in $\pi_1(F_b,e)^{ab}$. The proposition follows from this and equation $(\ref{eq2222})$.
\end{proof}

\begin{corollary}\label{cor nilp}
If $i_\#: \pi_1(F_e,e) \to \pi_1(E,e)$ is incejtive and $\pi_1(B,b)$ acts nilpotently on $\pi_1(F_b,e)^{ab}$, then $\pi_1(F_b,e)$ acts nilpotently on the first cohomology group $H^1(F_b,\Q)$ of the fiber.
\end{corollary}
\begin{proof}
By the previous corollary $\pi_1(B,b)$ acts nilpotently on $H_1(F_b,\Z)$. This implies that the "dual" action of $\pi_1(B,b)$ on the $\Q$-vector space $\Hom_\Z(H_1(F_b,\Z),\Q)$ is nilpotent, since the natural isomorphism $H^1(F_b, \Q) \simeq  \Hom_\Q(H_1(F_b, \Z),\Q)$ is compatible to the actions of $\pi_1(B,b)$ coming from homotopy equivalences, we deduce that $\pi_1(B,b)$ acts nilpotently on $H^1(F_b, \Q) $.
\end{proof}
\begin{corollary}\label{cor RFC1}
If $i_\#: \pi_1(F_e,e) \to \pi_1(E,e)$ is injective, $\pi_1(B,b)$ acts nilpotently on $\pi_1(F_b,e)^{ab}$ and $H^*(F_b,\Q)=0$ for $k>1$, then $\pi_1(B,b)$ acts nilpotently in the cohomology of the fiber and $F\to E \to B$ is a rational fibration.
\end{corollary}
 \begin{proof}
The corollary follows from the previous one and from theorem \ref{Th RF}.
\end{proof}
\begin{remark}\label{rmk conj}
Under the assumption $i_\#: \pi_1(F_e,e) \to \pi_1(E,e)$ is injective, the condition $\pi_1(B,b)$ acts nilpotently on $\pi_1(F_b,e)^{ab}$ is equivalent to the statement: the action of $\pi_1(E,e)$ on $\pi_1(F_P,e)^{ab}$ induced by conjuguacy on $\pi_1(F_P,e)$ (regarded as a normal subgroup) is nilpotent.
\end{remark}

\section{Orbit configuration spaces of surfaces}\label{S2}
We start the section by recalling general facts on orbit configuration spaces (definition, fibration theorem) and fix some notations. Then, we introduce the orbit configuration spaces of surfaces ($C_n^G(S)$) that we will consider in this text: $S$ is a cofinite set in a closed (boundaryless and compact) orientable surface $\bar{S}$ and $G$ is a finite group acting freely by homeomorphisms on $S$. In subsection \ref{S21}, we consider homotopy groups, finiteness of Betti numbers of such orbit configuration spaces and give examples of such spaces. In subsection \ref{S22}, we show that classical fibrations associated to orbit configuration spaces of surfaces we consider are rational fibrations, and establish main result $(1)$ announced in the introduction of the paper. We consider the cross-section problem for fibrations associated to $C_n^G(S^2)$, in the last subsection. Generally, cross-sections exists (with some exceptions). The cross-sections will be used in section \ref{S3}.\\\\
Let $X$ be a topological space and $G$ be a group acting on $X$ by homeomorphisms, the orbit configuration space of $n$-points on $X$ is the topological space:
$$C_n^G(M)=\{(p_1,\dots,p_n)\in X^n \vert p_i \notin G\cdot p_j \:\:\text{for $i\neq j$}\},$$
where $G\cdot x$ denotes the orbit of $x \in X$ under $G$. For $G=\{1\}$, $C_n^G(X)$ is the classical configuration space and will be denoted by $C_n(S)$.
\begin{theorem}[\cite{Xico}]
Assume that $M$ is a boundaryless manifold and $G$ is a finite group acting freely by homeomorphisms on $M$. For $n\geq k\geq 1$, the projection on the first $k$ coordinates $\pi_{n,k}:C_n^G(M)\to C_k^G(M)$ is a locally trivial fibration.
\end{theorem}
The fiber of the projection $\pi_{n,k}$ over $P=(p_1,\cdots,p_k)\in C_k(G)$ correspond to:
$$M\setminus O_k(P), \:\: \text{where $O_k(P)=\cup_{i\in [1,k]} G\cdot p_k$}.$$
For convenience we set $\pi^n:= \pi_{n,n-1}$ for $n\geq 1$ and the notation $O(P)$ will be used in the sequel.\\\\
Throughout the rest of the text, we consider configuration spaces $C_n^G(S)$ where $G$ is a finite group acting freely (by homeomorphisms) on $S$, where $S$ correspond to a compact connected boundaryless oriented surface $\bar{S}$ minus a finite number of punctures (eventually none). The fiber $F_P$ of $\pi_n$ over $P\in C_n^G(S)$ is homeomorphic to $S\setminus O_n(P)$ with $O_n(P) \neq \empty$ and hence:
\begin{equation}\label{Fiber HH}
\pi_k(F_P)=H^k(F_p,\Q)=0,\quad \text{for $k>1$}.
\end{equation}
For the sequel, we fix a point $P\in C_n^G(S)$, conserve the notation $F_P$ and denote by $S^2$ the two sphere. We will sometimes assume that $S$ is not homeomorphic to $S^2$. In that case we eliminate two types of configuration spaces (Cf. Subsection \ref{S21}).
\subsection{Examples, finiteness of Betti numbers and homotopy groups}\label{S21}
Take $\bar{S},S$ and $G$ as in the previous paragraph. The configuration spaces $C_n^G(S)$ with $G$ a group acting by orientation preserving homeomorphisms on $S$ with $\bar{S}=S^2$, can be classified and are restrictions of actions of finite subgroups of $\mathrm{SO}(3)$ on $S^2$. Now consider the case of $G$ acting on $S=S^2$. Since the action is free the quotient map $S^2\to S^2/G$ is a covering map and $S^2/G$ is a closed surface of Euler characteristic equal to $\frac{2}{G}$. Hence, either $\vert G\vert =1$, either $\vert G \vert =2$ and $S^2/G$ is the $2$-dimensional real projective plane. In the latter case, the action of $G$ is equivalent to the antipodal action. We have shown that up to equivalence $C_n^G(S^2)$ is either the classical configuration space $C_n(S^2)$ or the orbit configuration space $C_n^{\Z/2\Z}(S^2)$, where the action of $\Z/2\Z$ is the antipodal action. In general, free actions in the case $\bar{S}=S^2$ come from finite isometry groups of $S^2$.\\\\
If $\bar{S}=S_g$ a genus $g$ closed orientable surface with $g\geq 1$, then the action of the rotation group of order $2g$ of a $2g$ regular gone induces an action $\Z/2g\Z$ over $\bar{S}$ with exactly $g+2$ points $p_1,\dots, p_{g+2}$ with non-trivial stabilizer. Hence, we can construct an orbit configuration space of the type $C_n^{G}(S)$ for $G=\Z/2g\Z$ and $S=\bar{S}\setminus \{p_1,\dots,p_{g+2}\}$. For $g> 2$, one can obtain a free action of $\Z/(g-1)\Z$ on $S_g$ from a rotation of order $g-1$ of $\R^3$ acting on an embedding $S_g\subset \R^3$, by disposing $g-1$ "holes" of $S_g$ symmetrically around a given "hole". This allows to construct orbit configuration spaces $C_n^{\Z/(g-1)\Z}(S_g)$ ($S=\bar{S}$).% One can also construct a action of $\Z/2\Z$ with exactly $2g$ points with non-trivial stabilizers using a rotation of order $2$ of $\R^3$ and an appropriate embedding of $S_g$ into $\R^3$. A free action of $\Z/2\Z$ on $S_g$ can be obtained using $-\mathrm{id}$ of $\R^3$ if one embeds $S_g$ in a symmetric manner with respect to the origin of $\R^3$. \\\\
In the sequel, we go back to the general situation: $\bar{S}$ is a closed, orientable surface $S\subset \bar{S}$ is a cofinite set and $G$ acts freely on $S$, with no more assumptions except if we mention so.
\begin{proposition}\label{prop betti}
The space $C_n^G(S)$ has finite Betti numbers.
\end{proposition}
\begin{proof}
Since $C_n^G(S)$ is an orientable manifold, we only need to show that the cohomolgy with compact support of $X$ is finite dimensional. Let $X$ be a locally compact space, $U$ an open subset of $X$ and $V=X\setminus U$. If two of the three spaces $X,U,V$ have finite dimensional cohomology with compact support, then the third does. Indeed, there cohomology with compact support are related by a long exact sequence (\cite{Iver}, equation 7.6, p. 185). Using this argument one shows by induction on $l$ that if $V=\cup_{k=1}^l V_k$, where $V_k$ are closed and the intersections $V_{k_1}\cap \cdots \cap V_{k_m}$ (for $m\geq 1$) has finite dimensional cohomology with compact support, then $V$ also has finite dimensional cohomology with compact support. Now notice that $C_n^G(S)=(S)^n \setminus D$, where $D=\cup_{k=1}^l D_{\alpha_{k}}$ with $D_{\alpha_k}$ a set of points $(p_1,\dots, p_n) \in (S)^n$ given by the equation $p_{i_{\alpha_k}}=g_{\alpha_k}\cdot p_{j_{\alpha_k}}$. For $l\geq 1$, the intersections of $D_{\underline{k}}:=D_{\alpha_{k_1}}\cap \cdots \cap D_{\alpha_{k_l}}$ are homoemorphic to $(S)^{m_{\bar{k}}}$ for a given $m_{\bar{k}} \geq 0$ ($(S)^0=\emptyset)$) and hence have finite dimensional cohomology with compact support. This proves that the closed set $D$ has finite dimensional cohomology with compact support. The proposition follows, since $S^n$ has finite dimensional with compact support and $C_n^G(S)=(S^n)\setminus D$. 
\end{proof}
\begin{proposition}
If $S$ is not homeomorphic to $S^2$, then the space $C_n^G(S)$ is aspheric (a $K(\pi,1)$ space).
\end{proposition}
\begin{proof}
The surface $S$ is aspheric and the surfaces of the form $S\setminus O_k(P)$ for $O_k(P)$ corresponding to fibers of the fibrations $\pi^k:C_{k+1}^G(S)\to C_k^G(S)$ are also aspheric. Using this and the homotopy long exact sequence of the fibration $\pi^{n}$, one can show by induction that $C_{n}^G(S)$ is aspheric for all $n\geq 1$.
\end{proof}
\begin{corollary}\label{cor inj}
Choose a $P'\in F_P$. If $S$ is not homeomorphic to $S^2$, then the map $i_\#: \pi_1(F_b,P')\to \pi_1(C_{n+1}^G(S),P')$ induced by the inclusion of the fiber is injective.
\end{corollary}
\begin{proof}
This follows from the exact sequence of the fibration, since the base is aspheric.
\end{proof}
\subsection{Rationality of fibrations of orbit configuration spaces of surfaces}\label{S22}
We recall that $\bar{S}$ is a closed oriented surface $S\subset \bar{S}$ is a cofinite set and $G$ acts freely on $S$.
\begin{proposition}\label{prop nil S}
Assume that $S$ is compact (i.e. $S=\bar{S}$ ) and not homeomorphic to $S^2$. The fundamental group of $C_n(S)$ acts nilpotently in the (rational) cohomology of the fiber of $\pi^n:C_{n+1}(S)\to C_n(S)$.
\end{proposition}
\begin{proof}
For convenience we prove the proposition for $\pi'_n$ obtained by projecting on the last (instead of the first) $n$ coordinates. Fixe a base point $P\in C_n(S)$ and $P'$ in the fiber $F_P$ of $\pi'_n$ over $P$. In virtue of corollary \ref{cor inj}, equation (\ref{Fiber HH}), corollary \ref{cor nilp} and remark \ref{rmk conj}, we only need to show that the action of $\pi_1(C_{n+1}(S),P')$ on $\pi_1(F_P,P')^{ab}$ induced by conjugacy on $\pi_1(F_P,P')$ (regarded as a normal subgroup) is nilpotent for a given $P'$. It is proved in \cite{GuaSP} (see corollary 8) that $\pi_1(C_{n+1}(S),P')$ for a given $P'$ is generated by some loops $\rho_{i,j},\tau_{i,j}$ for $i\in [1,n+1]$ and $j\in [1,g]$, where $g$ is the genus of $S$. For a fixed $i$ these loops correspond to the standard generators of the homology of the $i$-th copy of $S$ in $S^{n+1}$ containing $C_n(S)$. In loc. cit., the authors also introduce loops $C_{i,j}$ for $i <j$ in $[1,n+1]$ in $\pi_1(C_{n+1}(S),P')$ (one can regard these loops as the classical braids). One can check that, the group $\pi_1(F_P,P')$ is freely generated by the loops $C_{2,n+1},\dots,C_{n,n+1}, \rho_{n+1,j}, \tau_{n+1,j}$ for $j\in[1,g]$ and contains the loop $C_{1,n+1}$. Relations in corollary 8 of loc. cit. (index by $(I), (II), \dots, (XIX)$) suffice to describe the action of $\pi_1(C_{n+1}(S),P')$ on $\pi_1(F_P,P')^{ab}$ we are interested in. We give equations determining this action. These equations are easily obtained from the relations in \cite{GuaSP} and we omit the proof. We will only indicate next to each equation the index of the relations in \cite{GuaSP} implying it. For $x \in \pi_1(F_P,P')$ we denote by $[x]$ the class of $x$ in the abelianization.
For $k,l\in[1,n]$ and $r,s \in [1,g]$, we have:
$$ \rho_{k,r}\cdot [\rho_{n+1,s}]=[\rho_{n+1,s}], \qquad (IV),(VIII), (XII)$$

\begin{displaymath}
\rho_{k,r}\cdot [\tau_{n+1,s}]= \left\{
\begin{array}{lr}
\: [\tau_{n+1,s}] & \text{if $r\neq s$}\\
\: [\tau_{n+1,s}]+ [C_{k,n+1} ] & \text{otherwise}
\end{array}
\right. \qquad (II), (VII), (XV)
\end{displaymath}

$$ \rho_{k,r}\cdot [C_{l,n+1}]=[C_{l,n+1}] \qquad (X), (XVII) , (XIX)$$
and
$$ \tau_{k,r}\cdot [\tau_{n+1,s}]=[\tau_{n+1,s}], \qquad (VI), (IX), (XIII)$$
\begin{displaymath}
\tau_{k,r}^{-1}\cdot [\rho_{n+1,s}]= \left\{
\begin{array}{lr}
\: [\rho_{n+1,s}] & \text{if $r\neq s$}\\
\: [\rho_{n+1,s}] +[C_{k,n+1}]& otherwise
\end{array}
\right. \qquad (III), (V), (XIV)
\end{displaymath}
$$ \tau_{k,r}\cdot [C_{l,n+1}]=[C_{l,n+1}] \qquad (XI), (XVI), (XVIII) $$
Since $\pi_1(F_P,P')$ is generated by $C_{1,n+1},\dots, C_{n,n+1}, \rho_{n+1,s}, \tau_{n+1,s}$ for $s\in[1,g]$, these equations describe entirely the action of $\pi_1(C_{n+1}(S), P')$. Indeed, the remaining generators of $\pi_1(C_{n+1}(S),P')$ not appearing in the equations are $\rho_{n+1,r}$ and $\tau_{n+1,r}$ and they act trivially. We deduce from this and the equations that the following subgroups of $\pi_1(F_P,P')^{ab}$ are stable under $\pi_1(C_{n+1}(S),P')$:
$$ 0=U(0)\subset U(1) \subset U(2)=\pi_1(F_P,P')^{ab},$$
where $U(1)$ is the subgroup of generated by $[C_{1,n+1}],\dots,[C_{n,n+1}]$, and that $\pi_1(C_{n+1}(S),P')$ acts trivially on $U(m+1)/U(m)$ for $m\in[0,1]$. Hence, the action is nilpotent. We have proved the proposition.
\end{proof}
\begin{lemma}
The fundamental group of $C_n(\R^2)$ acts trivially in the (rational) cohomology of the fiber of $\pi^n:C_{n+1}(\R^2) \to C_n(\R^2)$.
\end{lemma}
\begin{proof}
One can prove this using Artin pure braid relations and by adapting the proof of the previous proposition.
\end{proof}
\begin{proposition}\label{prop nil S2}
The fundamental group of $C_n(S^2)$ acts trivially in the (rational) cohomology of the fiber of $\pi^n: C_{n+1}(S^2) \to C_n(S^2)$.
\end{proposition}
\begin{proof}
We regard $S^2$ as the one point compactification $\R^2 \cup \infty$ of $\R^2$. Denote by $f_n$ the inclusion $C_n(\R^2) \to C_{n+1}(S^2)$ given by $(p_1,\dots,p_n)\mapsto (\infty, p_1,\dots,p_n)$ (with the convention $C_0(\R^2)$ is a point and $f_0$ maps the point to $\infty$). The map $f$ is a homeomorphism between $C_n(\R^2)$ and the fiber $F^1_\infty$ of $\pi_{n,1}:C_{n+1}(S^2) \to C_1(S^2)=S^2$ over $\infty$. Since $S^2$ is simply connected, it follows from the homotopy long exact sequence of $\pi_{n,1}$ that $f_n$ induces a surjective morphism on fundamental groups. On the other hand, the fibration $\pi^{n-1}_{\R^2}:C_n(\R^2) \to C_{n-1}(\R^2)$ ($C_0(\R^2)$ is a point) is equivalent to the pullback by $f_{n-1}$ of $\pi^n:C_{n+1}(S^2) \to C_n(S^2)$. In particular, the action $(a1)$ of the fundamental group of $C_{n-1}(\R^2)$ in the cohomology of the fiber $F_{\R^2}$ of $\pi^{n-1}_{\R^2}$ factors via $f_{n-1}$ through the action $(a2)$ of the fundamental group of $C_n(S^2)$ on the fiber of $\pi^n$. Since the action $(a1)$ is trivial (Cf. previous lemma) and the map induced by $f_{n-1}$ on fundamental groups is surjective, we deduce that $(a2)$ is trivial. We have proved the proposition.
\end{proof}

\begin{proposition}\label{prop nilp act}
The fundamental group of $C_n^G(S)$ acts nilpotently in the (rational) cohomology of the fiber of $\pi^n:C_{n+1}^G(S)\to C_n^G(S)$.
\end{proposition}
\begin{proof}
Let $g_1,\dots,g_m$ be the elements of $G$ and $q_1,\dots,q_l$ be the elements of $\bar{S} \setminus S$ ($l$ is eventually null). We have a well-defined continuous map:
\begin{align}\label{pullback}
\begin{split}
f_n &: C_n^G(S) \to C_{mn+l}(\bar{S}) \\
&(p_1,\dots,p_n) \to (q_1,\dots,q_l, \alpha_1, \dots, \alpha_{mn}),
\end{split}
\end{align}
where $\alpha_{k}=g_{r+1}\cdot p_{s+1}$ if $k=sm+r$ with $r\geq 0$. Essentially, the coordinates of $f_n(p_1,\cdots,p_n)$ are exactly the points $q_1,\dots,q_l$ and $g_r \cdot p_s$ for $(r,s)\in [1,m]\times [1,n]$. A similar map was considered in \cite{DCoh}, for $S=\C^\times$ and $\bar{S}$ (not closed as we impose here) equal to $\C$ and $G$ is a group of roots of unity acting by multiplication. The fibration $\pi^n:C_{n+1}^G(S) \to C_n^G(S)$ is equivalent to the pullback by $f_n$ of $\pi^{mn+l-1}_{\bar{S}}:C_{mn+l}(\bar{S}) \to C_{nm+l-1}(\bar{S})$. Hence, the action $(a1)$ of the fundamental group of $C_n^G(S)$ on the cohomology of the fiber $F_n$ of $\pi^n$ factors thought the action $(a2)$ of the fundamental group of $C_{nm+l-1}(\bar{S})$ on the cohomology of the fiber (homeomorphic to $F_n$) of $\pi_{\bar{S}}^{nm+l-1}$. Hence, since $(a2)$ is nilpotent by propositions \ref{prop nil S} and \ref{prop nil S2}, $(a1)$ is also nilpotent. We have proved the proposition.
\end{proof}
 \begin{corollary}
For $n>k$, the fibration $\pi_{n,k}:C_n^G(S) \to C_k^G(S)$ obtained by projecting on the first $k$ coordinates is a rational fibration. 
\end{corollary}
\begin{proof}
 It follows form proposition \ref{prop nilp act} and corollary  \ref{cor composition} of proposition \ref{prop composition} using a composition argument that $\pi_1 C_k^G(S)$ acts nilpotently in the homology groups of the fiber of $\pi_{n,k}$. Indeed, $\pi_{n,k}= \pi^{n-1} \circ \pi^{n-2} \circ \cdots \circ \pi^k$. Since $C_k^G(S)$ has finite Betti numbers (proposition \ref{prop betti}), $\pi_{n,k}$ is a rational fibration by \ref{Th RF}.
\end{proof}
Note that the fiber of the locally trivial fibration $\pi_{n,k}: C_n^G(S) \to C_k^G(S)$ is the orbit configuration space $C_{n-k}^G(S_{G,k})$, where $S_{G,k}$ is $S$ with $k$ orbits removed. 
\begin{corollary}\label{cor model fib}
For $k<n$, the space $C_{n}^G(S)$ has a Sullivan model $A_{n,k}=\Lambda V_{C_k^G(S)}\otimes \Lambda V_{C_{n-k}^G(S_{G,k})}$ fitting in a cdga sequence: $\Lambda V_{C_k^G(S)}\to A_{n,k} \to \Lambda V_{C_{n-k}^G(S_{G,k})}$, where $\Lambda V_X$ denotes the minimal model of $X$, the first morphism is $\mathrm{id}\otimes 1$ and the second is obtained by applying the augmentation map of $\Lambda V_{C_k^{G}(S)}$.
\end{corollary}
\begin{proof}
The corollary follows from the construction of a model for $C_{n}^G(S)$ using $\pi_{n,k}$ as in subsection \ref{S13} (Cf. the paragraph containing diagram (C) and proposition \ref{prop model RF} ), since $\pi_{n,k}$ is a rational fibration.
\end{proof}
The last $2$ corollaries correspond to main result $(1)$ announced in the introduction.

\subsection{Cross-sections for fibrations of orbit configuration spaces of the $2$-sphere}\label{S23}
We recall that (Cf. Subsection \ref{S21}) up to equivalence $C_n^G(S^2)$ is either $C_n(S^2)$ or $C_n^{\Z/2\Z}(S^2)$ where the action of $\Z/2\Z$ is the antipodal action.\\\\
The $2$-sphere $S^2$ is homeomorphic to $\PP^1$ the complex projective line. Since $3$ distinct points of $\PP^1$ define a projective basis of $\PP^1$, the group $\PGL(\C^2)$ of linear projective transformations of $\PP^1$ acts freely transitively on $C_3(\PP^1)\simeq C_3(S^2)$ and we have a homeomorphism $T: C_3(\PP^1)\to \PGL(\C^2)$ sending $(p_1,p_2,p_3)$ to the unique homography $T(p_1,p_2,p_3)$ mapping $(p_1,p_2,p_3)$ to $(0,1,\infty)$. In particular, for $n\geq 3$ we have and homeomorphism:
\begin{align}\label{Split Fib}\begin{split}
C_n(\PP^1)&\to C_3(\PP^1)\times C_{n-3}(\PP^1\setminus \{0,1,\infty\})\\ (p_1,\dots, p_n) &
\mapsto ((p_1,p_2,p_3), (h_{\underline{p}}(p_4),\dots,h_{\underline{p}}(p_n)))\end{split},\end{align}
where $h_{\underline{p}}=T(p_1,p_2,p_3)$ and $C_0(X)$ correspond to a point. Recall that we have a homotopy equivalence between $\PGL(\C^2)$ and $\PSU(\C^2)$ and isomorphisms-homeomorphisms $\PSU(\C^2)\simeq \mathrm{SO}_3(\R) \simeq \mathbb{P}^3(\R)$, where $ \mathbb{P}^3(\R)$ is the three-dimensional real projective space. Hence, $C_3(\PP^1)$ is homotopy equivalent to $\mathbb{P}^3(\R)$.\\\\
Denote by $T_pS^2 \subset \R^3$ the affine $2$-plane tangent to $S^2$ at $p$ and by $L_p:S^2 \setminus \{p\} \to T_{-p}S^2$ the stereographic projection with respect to $p$ onto $T_{-p}S^2$. The map $-L_p$ identifies $S^2\setminus \{p\}$ with $T_pS^2$, and we have a homeomorphism $L:C_2(S^2)\to TS^2, (x,y)\mapsto (x,-L_x(y))$, where $TS^2$ is the tangent bundle of $S^2$. The map $L$ restricts to a homeomorphism $C_n^{\Z/2\Z}(S^2)\to TS^2\setminus s_0(S^2)$, where $s_0:S^2\to TS^2$ is the zero cross-section. The space $TS^2\setminus s_0(S^2)$ is homotopy equivalent to the unitary bundle $US^2$ of $S^2$. Since the natural action of $\mathrm{SO}(3)$ on $S^2$ induces a free transitive action on $US^2$ (taken with respect to the standard metric on $S^2$), $US^2$ is homeomorphic to $\mathrm{SO}_3\simeq \PP^3(\R)$. We have proved that we have a homotopy equivalence:
\begin{equation}\label{eq hom eq Z2}
C_2^{\Z/2\Z}(S^2)\sim \PP^3(\R)
\end{equation}
\begin{lemma}\label{lem secc}
The fibration $\pi^n:C_{n+1}(\PP^1\setminus \{0,1,\infty\}) \to C_n(\PP^1 \setminus \{0,1,\infty\})$ admits a cross-section.
\end{lemma}
\begin{proof}
If one replaces $\PP^1\setminus \{0,1,\infty\}$ with the homeomorphic space $\C\setminus \{0,1\}$, the map $(z_1,\dots,z_n)\mapsto(z_1,\dots,z_n, 1+\vert z_1\vert +\cdots \vert z_n \vert)$ defines a cross-section.
\end{proof}
In general fibrations of configuration spaces of manifolds with punctures admit cross-sections.
\begin{proposition}\label{prop cross-section}
We consider the fibration $\pi^n: C_{n+1}^G(S^2)\to C_n^G(S^2)$. If $G=\{1\}$, then $\pi^n$ admits a cross-section if and only if $n\neq 2$. If $G\neq \{1\}$, then the fibration $\pi^n$ admits a cross-section if and only if $n\geq 2$.
\end{proposition}
\begin{proof}
We first consider the case $G=\{1\}$, i.e. the classical configuration space $C_n(S^2)$. It follows from the previous lemma and the homeomorphism in (\ref{Split Fib}) that $\pi^n:C_{n+1}(S^2)\to C_n(S^2)$ admits a cross-section for $n\geq 3$. In the case $n=2$, there is no cross-section, since $H_2(S^2,\Z)\simeq \Z$ can not inject into $H_2(C_3(\PP^1),\Z)\simeq H_2(\PP^3(\R),\Z)=0$. Finally, for $\pi^1:C_2(S^2)\to C_1(S^2)$, the map $S^2\to S^2, x_1\mapsto (x_1,-x_1)$, where $-x_1$ is the antipodal point to $x_1$ defines a cross-section. This proves the case $G=\{1\}$. We assume that $G\neq \{1\}$ and we can assume that action is the antipodal action. The map $f_n:C_n^G(S^2)\to C_{n\vert G \vert} (S^2)$ defined by (\ref{pullback}) in the proof of proposition \ref{prop nilp act} gives a bundle equivalence between $\pi^n:C_{n+1}^G(S^2) \to C_n^G(S^2)$ and the pullback by $f_n$ of $\pi^{n\vert G \vert }:C_{n\vert G \vert +1}(S^2)\to C_{n\vert G\vert}(S^2)$. Hence, it follows from the case $G=\{1\}$ that $\pi^n:C_{n+1}^G(S^2) \to C_n^G(S^2)$ admits a cross-section for $n\geq 2$. It follows from the homotopy equivalence in equation (\ref{eq hom eq Z2}) that $H^2(S^2)$ can not inject into $H^2(C_n^{\Z/2\Z}(S^2))$ and hence $\pi^2:C_2^{\Z/2\Z}(S^2)\to S^2$ do not admit cross-sections. We have proved the proposition.
\end{proof}
\begin{corollary}\label{prop cross1}
The fibration $\pi_{n,k} : C_{n}^G(S^2) \to C_k^G(S^2)$ admits a cross-section in the following cases: $n\geq k \geq 3$ and $G=\{1\}$, $n=2$ and $G=\{1\}$, $n\geq k \geq 2$ and $G\neq \{1\}$.
\end{corollary}
The cross-section problem for classical configuration spaces of the $2$-sphere (and closed surfaces) has already been considered in the literature. We constructed the cross-sections, for the case of the $2$-sphere, to avoid the reader going through other references.
\section{$\psi$-homotopy groups and minimal models of orbit configuration spaces of surfaces}\label{S3}
In this section, and as in the previous one, $\bar{S}$ is a closed orientable surface, $S\subset \bar{S}$ is a cofinite set and $G$ is a finite group acting freely by homeomorphisms on $S$. Here, we establish main results $(2)$, $(3)$ (the part not concerning Malcev Lie algebras) and $(4)$ announced in the introduction of the paper. We proved, in subsection \ref{S22}, the existence of Sullivan models $A_{n,k}$ ($k<n$) for $C_{n}^G(S)$ (corollary \ref{cor model fib}). We show, that except for some cases when $S\simeq S^2$ these models are minimal.%and describe the minimal model in the exceptional cases. 
 We also show that $C_n^G(S)$ is a rational $K(\pi,1)$ if $S$ is not homeomorphic to $S^2$ and compute the hotompy space in degree $k\geq 2$ of $C_n^G(S^2)$. It follows from the computation that $C_n^G(S^2)$ is not a rational $K(\pi,1)$. We first prove the results for $S$ not homeomorphic to $S^2$ (Subsection \ref{S31}) then consider the case $S=S^2$ (Subsection \ref{S32}).\\\\
We will fix some notations and recall elements from the previous section. For $X$ a topological space, the minimal model of $X$ is designated by $(\Lambda V_X,m_X)$. As in the previous section $S_{G,k}$ denoted $S$ with $k$ orbits removed. For $k>n$, the fiber of $\pi_{n,k}:C_{n}^G(S) \to C_k^G(S)$ correspond to $C_{n-k}^G(S_{G,k})$. We have shown (Subsection \ref{S22}, corollary \ref{cor model fib}) that $C_{n}^G(S)$ has a Sullivan model $A_{n,k}\simeq \Lambda V_{C_k^G(S)}\otimes \Lambda V_{C_{n-k}^G(S_{G,k})}$ fitting in a cdga sequence:
\begin{equation}\label{eq model Cn+1}
\Lambda V_{C_k^G(S)}\to A_{n,k}=\Lambda V_{C_k^G(S)}\otimes \Lambda V_{C_{n-k}^G(S_{G,k})} \to \Lambda V_{C_{n-k}^G(S_{G,k})},
\end{equation}
where the first morphism is $\mathrm{id}\otimes 1$ and the second is obtained by applying the augmentation map of $\Lambda V_{C_k^{G}(S)}$. 
\subsection{The case $S$ not homeomorphic to $S^2$}\label{S31}
In this subsection we assume that $S$ is not homeomorphic to $S^2$. 
\begin{proposition}\label{prop SRK}
The surface $S$ is a rational $K(\pi,1)$, i.e the minimal model of $S$ is the $1$-minimal model of $S$ (or equivalently $\pi_\psi^k(S)=0$ for $k\geq 2$).
\end{proposition}
\begin{proof}

The surface $S$ is either homeomorphic to $\R^2$, either homotopy equivalent to a bouquet of circles or is closed and orientable of genus $g\geq 1$, and has finite Betti numbers. If $S$ is homeomorphic to $\R^2$ then the minimal model of $S$ is $\Lambda 0$ and the $\psi$-homotopy space of $S$ correspond to $0$. In the other cases, one can describe $\Lambda V_S$ (\cite{FelHar2}, §8.3 and §8.5), and $V_S$ is concentrated in degree $1$. This proves the proposition. Another way to prove this is to show that the cohomolgy ring of $S$ is Koszul, for instance using a PBW basis argument (\cite{Prid} theorem 5.3). Then conclude using theorem \ref{Koszul KP1}, since $S$ is formal (\cite{FelHar} p. 156 and 162).
\end{proof}

\begin{proposition}\label{prop RKPS}
For $n\geq 1$:
\begin{itemize}
\item[1)] The space $C_{n}^G(S)$ is a rational $K(\pi,1)$, i.e. $\pi_\psi^k(C_{n}^G(S))=0$ for $k\geq 2$.
\item[2)] The Sullivan model $A_{n,k} $ of $C_{n}^G(S)$ in (\ref{eq model Cn+1}) is minimal and $1$-minimal.
\end{itemize}
\end{proposition}
\begin{proof}
We recall that $\pi^k$ is the fibration $C_{k+1}^G(S) \to C_k^G(S)$. We have already shown that $C_n^G(S)$ has finite Betti numbers (proposition \ref{prop betti}). It follows from the previous proposition (from this subsection), that the base space of $\pi^2$ and the fibers of $\pi^k$ are rational $K(\pi,1)$ spaces. We prove $(1)$ by induction using these facts and the exact sequence (\ref{ex seq 2}) of subsection \ref{S13}. Since $A_{n,k}=\Lambda V_{C_k^G(S)}\otimes \Lambda V_{C_{n-k}^G(S_{G,k})}\simeq \Lambda (V_{C_k^G(S)}\oplus V_{C_{n-k}^G(S_{G,k})})$ and $\pi_\psi^*(X)=V_X$, we deduce from $(1)$ that $A_{n,k} \simeq \Lambda V$ with $V$ concentrated in degree $1$. This proves that $A_{n,k}$ is minimal. %Hence, one can show $(1)$ by induction using the exact sequence (\ref{ex seq 2}) of subsection \ref{S13}. In particular, the base $C_{n}^G(S)$ of $\pi^n$ is a rational $K(\pi,1)$ this implies that the connecting homomorphism in (\ref{ex seq 2}) is $0$ and hence the model in the proposition is minimal by theorem \ref{th minimal RF}. This proves $(2)$. We have proved the proposition.
\end{proof}
We have proved main result $(2)$ for $S$ not homeomorphic to $S^2$ and the part not concerning Malcev Lie algebras of main result $(3)$.
\subsection{The case $S=S^2$}\label{S32}
We recall that up to equivalence (Cf. subsection \ref{S21}) $C_n^G(S^2)$ is either the configuration space $C_n(S^2)$ either $C_n^{\Z/2\Z}(S^2)$, where the action of $\Z/2\Z$ is the antipodal action.
\begin{proposition}\label{prop model S2}
\begin{itemize}
\item[1)] The minimal model of $S$ is $\Lambda V_{S^2}$, where $V_{S^2}= (\Q x)_2 \oplus (\Q y)_3$ (indices for degree), with differential given by $dx=0, dy=x^2$; a filtration of Sullivan algebras can be given by: $V_{S^2}(0)=\Q \cdot x $, $V_{S^2}(k)=V_{S^2}$ for $k\geq 1$.
\item[2)] The $\psi$-homotopy space of $S^2$ is $(\Q)_2\oplus(\Q)_3$, where indices correspond to the degree.
\end{itemize}
\end{proposition}
\begin{proof}
The construction of $\Lambda V_{S^k}$ can be found in \cite{FelHar} (example 1, p. 142). Recall that if $\Lambda V$ is minimal then $\pi_\psi^*(\Lambda V)$ correspond to $V$ and hence the $\psi$-homotopy space of $S^2$ correspond to $(\Q)_2\oplus(\Q)_3$. We have proved the proposition.
\end{proof}
\begin{proposition}\label{prop models S2}
Take $n\geq 2$:
\begin{itemize}
\item[1)] 
\begin{itemize} 
\item[a)] For $G=\{1\}$, the Sullivan model $A_{n,k}$ of $C_{n}^G(S^2)$ as in (\ref{eq model Cn+1}) is minimal if and only if $k\geq 3$ or $n=2$.
\item[b)] For $G\neq \{1\}$, the Sullivan model $A_{n,k}$ of $C_{n}^G(S^2)$ as in (\ref{eq model Cn+1}) is minimal if and only if $k\geq 2$.
\end{itemize}
\item[2)] The minimal models of $C_3(S^2)$ and $C_2^{\Z/2\Z}(S^2)$ are isomorphic (as Sullivan algebras) and correspond to $\Lambda \Q\cdot z$, where $z$ is of degree $3$. The differential is null and one can take the constant filtration equal to $\Q\cdot z$.
\item[3)] The space $C_n^G(S^2)$ is not rational $K(\pi,1)$ spaces. The $\psi$-homotopy space $\pi_\psi^{\geq 2}(C_n^G(S^2))$ in degre $k\geq 2$ correspond to $(\Q)_2\oplus (\Q)_3$ if $n=1$ or $(G,n)=(\{1\},2)$, and to $(\Q)_3$ otherwise.
\end{itemize}
\end{proposition}
\begin{proof}
We have proved in proposition \ref{prop cross1} that $\pi_{n,k}:C_{n}^G(S^2)\to C_k^G(S^2)$ admits a cross-section in the cases given in $(1.a)$ and $(1.b)$. The "if" part of $(1)$ follows from the existence of the cross-section by corollary \ref{sec RF}. We prove the "only if" part after proving $(2)$ and $(3)$. We now prove $(2)$. We have seen (subsection \ref{S23}) that $C_3(S^2)$ and $C_2^{\Z/2\Z}(S^2)$ are homotopy equivalent to the $3$-dimensional real projective space whose rational cohomology is $(\Q)_0\oplus (\Q)_3$ (indices for degree). Hence, by setting $dz=0$, the map $\Lambda \Q \cdot z\to A_{PL}(C_3(S^2))$ given by $z\mapsto \tilde{a}$ where $a$ is a lift of a non-trivial class in $H^3(A_{PL}(C_3(S^3)))\simeq H^3(C_3(S^3),\Q)$, gives a minimal model of $C_3(S^2)$. The same argument applies to $C_n^{\Z/2\Z}(S^2)$. We have proved $(2)$. We now prove $(3)$. By $(2)$ and the "if" part of $(1)$ the minimal model of $C_n^G(S^2)$ is:
$$ \Lambda V_{S^2} \text{ if $n=1$}, \quad \Lambda V_{S^2}\otimes \Lambda V_{\R^2} \text{ if $G=\{1\}$ and $n=2$},$$
and in the other cases it is of the form:
$$\Lambda \Q \cdot z \otimes \Lambda V_{C_{n-l}^G(S^2_{G,l})},$$
with $l=3$ if $G=\{1\}$ and $l=2$ if $G\neq \{1\}$, and where $z$ is of degree $3$. Hence, $\pi_\psi^*(C_n^G(S^2))$ is one of the spaces $V_{S^2}$, $V_{S^2} \oplus V_{\R^2}$ or $(\Q)_3\oplus V_{C_{n-l}^G(S^2_{G,l})}$ depending on $n$ and $G$. Now $V_{\R^2}=0$, $V_{S^2}=(\Q)_2\oplus (\Q)_3$ by the previous proposition, and $V_{C_{n-l}^G(S^2_{G,l})}$ is concentrated in degree $1$ (eventually null) by proposition \ref{prop RKPS}. Claim $(3)$ follows. We now prove the "only if" part of $(1)$. We first consider the case $(a)$ ($G=\{1\}$). If one of the Sullivan models $A_{n,2}$ or $A_{n,1}$ was minimal, then $\pi_{\psi}^*(C_n(S^2))$ will be isomorphic to $\pi_{\psi}^*(C_2(S^2))\oplus \pi_{\psi}^*(C_{n-2}(S^2_{\{1\},2})$ or $\pi_{\psi}^*(C_1(S^2))\oplus \pi_{\psi}^*(C_{n-1}(S^2_{\{1\},1})$, and $\pi_{\psi}^*(C_n(S^2))$ will have a non-trivial degree $2$ component ($\pi_{\psi}^*(C_2(S^2))$ and $\pi_{\psi}^*(C_1(S^2))$ do by $(3)$). This leads to a contradiction with $(3)$ for $n\geq 3$. Hence, for $n\geq 3$, $A_{n,2}$ and $A_{n,1}$ of $(a)$ are not minimal. This proves the "only if" part of $(1.a)$. The "only if" part of $(1.b)$ is obtained using the same reasoning. This completes the proof of $(1)$. We have proved the proposition.
\end{proof}
This completes the proof of main result $(2)$ and proves main result $(4)$.

\section{On minimal models and Malcev Lie algebras of orbit configuration spaces of surfaces}\label{S4}
In the first subsection, we recall general facts relating $1$-minimal models to Malcev Lie algebras, under reasonable assumption. This is done using a third part the Fundamental Lie algebra $L_{\Lambda V_X}$ of a minimal Sullivan algebra $\Lambda V_X$. For $X$ a path connected topological space with $H^1(X,\Q)$ finite dimensional, $L_{\Lambda V_X}$ is isomorphic to the Malcev Lie algebras $L_X$ of $\pi_1X$ and it determines entirely the $1$-minimal model of $X$. In particular, it determines the minimal model of $X$ if moreover $X$ is a rational $K(\pi,1)$, as the spaces $C_n^G(S)$ for $S$ not homeomorphic to $S^2$. This establishes the part concerning Malcev Lie algebras of main result $(3)$ and is used to prove main result $(6)$ (subsection \ref{S42}). In some cases, the Malcev Lie algebra of $C_n^G(S)$ is known and is isomorphic to the degree completion of a graded Lie algebra and the $1$-minimal model (and minimal model) can be recovered using a variant of the Cartan-Chevalley-Eilenberg cdga of Lie algebras (main result $(6)$). In subsection \ref{S43}, we prove main result $(5)$.

\subsection{$1$-minimal models and Malcev Lie algebras}\label{S41}
For $V$ a vector space and $U,V$ two vector subspaces of $V$, we denote by $U\wedge V$ and $U \wedge U$ the natural images of $U\otimes V$ and $U\otimes U$ in $\Lambda V$.
To a group $H$ one can associate functorially (\cite{Q1}) a filtered $\Q$-Lie algebra $L(H)$ called the Malcev Lie algebra of $H$ over $\Q$. We will denote by $\{F_i L(H)\}_{i\geq 1}$ the filtration of $L(H)$. Since we will use elements from \cite{FelHar2}, we note that the definition of a (minimal) Sullivan algebra \cite{FelHar2} is not exactly the one in \cite{FelHar} used here. But, the definitions are consistent as mentioned by the authors. It can be readily checked that (Minimal) Sullivan algebras we use here (following \cite{FelHar}) are (minimal) Sullivan algebras in the sense of \cite{FelHar2}. Moreover, minimal models in \cite{FelHar} are also unique. In particular, the minimal models from \cite{FelHar} and \cite{FelHar2} are isomorphic.\\\\
If $(\Lambda V,d)$ is minimal and $V_1$ is the degree $1$ part of $V$, then $d(V_1)\subset V_1\wedge V_1$. Important invariants of Sullivan algebras are their fundamental Lie algebra and homotopy Lie algebra. We will only use fundamental Lie algebras of minimal Sullivan algebras (Cf. \cite{FelHar} for definitions in the other cases) .
\begin{definition}[\cite{FelHar2}]
The fundamental Lie algebra of a minimal Sullivan algebra $(\Lambda V,d)$ is the dual space $V_1^\vee$ of $V_1$ with Lie bracket $d^\vee: V_1^\vee\to \wedge^2 V_1^\vee$ obtained by dualizing the restriction $V_1\to V_1\wedge V_1$ of $d$.
\end{definition}
\begin{remark}\label{rmk fund}
The fundamental Lie algebra $L_{(\Lambda V,d)}$ of a minimal Sullivan algebra $(\Lambda V, d)$ is also the fundamental Lie algebra of the $1$-minimal model $(\Lambda V_1,d)$ of $(\Lambda V,d)$ (Cf. remark \ref{rmk min 1min}).
\end{remark}
For $L$ a Lie algebra we denote by $\{\Gamma_i L\}_{i\geq 1}$ the lower central series of $L$: $\Gamma_1L=L$ and $\Gamma_{i+1}L=[L,\Gamma_i L]$ for $i\geq 1$. We will prove the following proposition in the appendix:
\begin{proposition}\label{prop HLie Mal}
If $X$ is a path connected topological space with $H^1(X,\Q)$ finite dimensional, then the fundamental Lie algebra of the minimal model (or equivalently of the $1$-minimal model) of $X$ is isomorphic to the Malcev Lie algebra $L_X$ of $\pi_1(X)$, and the filtration $\{F_iL_X\}_{i\geq 1}$ is the lower central series filtration $\{\Gamma_i L_X\}_{i\geq 2}$.
\end{proposition}
For $L$ a Lie algebra the Lie bracket gives a linear map $L^\vee \to L^\vee \wedge L^\vee$, where $L^\vee$ is the dual space of $L$. Since the bracket satisfies the Jacobi identity, the previous map extends to a differential $d_L$ on $\Lambda L^\vee$, where $L^\vee$ is seen as a graded space concentrated in degree $1$. This construction in know as Cartan-Chevalley-Eilenberg construction and the cdga $(\Lambda L^\vee, d_L)$ will also be denoted by $C^*(L)$. A morphism of Lie algebras $L\to M$ induces a cdga morphism $C^*(M)\to C^*(L)$.
\begin{proposition}\label{prop CCE model}
If $X$ is a path connected CW complex and $H^1(X,\Q)$ is finite dimensional, then $F_iL_X=\Gamma_i L_X$ for $i\geq 1$ and the Lie algebra $L_X$ entierly determines the $1$-minimal model of $X$. Moreover, the cdga $A_X=\varinjlim C^*(L_X/\Gamma_i L_X)$ is a $1$-minimal model of $X$.
\end{proposition}
\begin{proof}
If $V$ is concentrated in degree $1$, then a Sullivan algebra $A=\Lambda V$ is minimal. If moreover $H^1(A)$ is finite dimensional then $\varinjlim C^*(L_A/\Gamma_i L_A)=A$ (Cf. \cite{FelHar2}, remark p. 55). The Sullivan algebra $\Lambda (V_X)_1$ is the $1$-minimal model of $\Lambda V_X$. Hence, $L_{\Lambda V_X}=L_{\Lambda (V_X)_1}$ and $\varinjlim C^*(L_{\Lambda V_X}/\Gamma_i L_{\Lambda V_X})=\varinjlim C^*(L_{\Lambda (V_X)_1}/\Gamma_i L_{\Lambda (V_X)_1})=\Lambda (V_X)_1$. Indeed, $H^1(\Lambda (V_X)_1)=H^1(\Lambda V_X)=H^1(X)$. The proposition now follows from the previous one.
\end{proof}

\begin{remark}
See also \cite{MGRH} where it is assumed that $\pi_X$ is finitely presented. This is the case of the spaces $C_n^G(S)$ but we did not address the question.
\end{remark}
For $L=\oplus_{i\geq 1} L_i$ a graded Lie algebra (i.e. $[L_i,L_j] \subset L_{i+j}$), we denote by $C_{\mathrm{gr}}^*(L)$ the cdga obtained by replacing $L^\vee$ in the construction of $C^*(L)$ by the graded dual $L'$ of $L$. We denote by $\hat{L}=\prod_{i\geq 1} L_i$ the degree completion of $L$.
\begin{lemma}\label{lem filgrL}
Let $L$ be a graded Lie algebra as above, generated by $L_1$ which is finite dimensional. The vector subspace $\prod_{j\geq i} L_j \subset \hat{L}$ is equal to $\Gamma_i \hat{L}$.
\end{lemma}
\begin{proof}
Let $x_1,\dots,x_m$ be a basis of $L_1$. For any $k\geq 1$ we have $L_{k+1}=[L_1,L_k]$. Hence, an elment $z_j\in L_j$ with $j>i$ is of the form: 
$$ \sum_{\underline{\alpha} \in [1,m]^i} \mathrm{ad}(x_{\alpha_1})\circ\cdots \circ \mathrm{ad}(x_{\alpha_i})(z_{\underline{\alpha},j}),$$
for some elements $z_{\underline{\alpha},j}\in L_{j-i}$ and where $\underline{\alpha}=(\alpha_1,\dots, \alpha_i)$. Therefore, an element of $F_i \hat{L}$ is of the form 
$$z_i+ \sum_{\underline{\alpha} \in [1,m]^i} \mathrm{ad}(x_{\alpha_1})\circ\cdots \circ \mathrm{ad}(x_{\alpha_i})(\sum_{j>i} z_{\underline{\alpha},j}),$$
for some $z_i\in L_i$ and $z_{\underline{\alpha},j} \in L_{j-i}$. This proves the proposition.
\end{proof}
\begin{lemma}
Let $L$ be a graded Lie algebra as above, generated by $L_1$ which is finite dimensional, and let $\hat{L}$ be the degree completion of $L$. The cdga $\varinjlim C^*(\hat{L}/\Gamma_i \hat{L})$ is isomorphic to $C_{\mathrm{gr}}^*(L)$.
\end{lemma}
\begin{proof}
Since $L_1$ is finite dimensional and generates $L$, we have compatible, natural isomorphisms $L/\Gamma_i L\to \hat{L}/\Gamma_i \hat{L}$ (the maps are isomorphisms by lemma \ref{lem filgrL}) and for $i\geq 1$ and $L_i$ is finite dimensional ($i\geq 1$). Hence, $\varinjlim C^*(\hat{L}/\Gamma_i \hat{L})\simeq \varinjlim C_{\mathrm{gr}}^*(L/\Gamma_i L)$. The quotient maps $L\to L/\Gamma_iL$, induce compatible morphisms $C_{\mathrm{gr}}^*(L/\Gamma_i L) \to C_{\mathrm{gr}}^*(L)$. Hence, we have a natural morphism $\varinjlim C_{\mathrm{gr}}^*(L/\Gamma_i L) \to C_{\mathrm{gr}}^*(L)$. The morphism is an isomorphism, since $\varinjlim_i (\Lambda \oplus_{j<i} L_i' )\simeq \Lambda L'$ (' is for graded dual). We have proved the lemma.
\end{proof}
\begin{proposition}\label{prop CCE model}
Let $X$ be a path connected CW complex with $H^1(X,\Q)$ finite dimensional. Assume that $L_X$ is the degrre completion of a graded Lie algebra $\mathcal{L}_X$ generated by its finite dimensional degree $1$ component. The dga $C_{\mathrm{gr}}(\mathcal{L})$ is a $1$-minimal model of $X$. Moreover, if $X$ is a rational $K(\pi,1)$, then $C_{\mathrm{gr}}(\mathcal{L})$ is the minimal model of $X$.
\end{proposition}
\begin{proof}
The proposition follows from the previous lemma and the previous proposition.
\end{proof}
\begin{remark}\label{rmk CCE}
The assumption on $L_X$ of the proposition is equivalent to the fact that $L_X$ is the degree completion of its associated graded. This corresponds to the notion of filtered-formality of \cite{SucWan}, which holds for instance for configuration spaces of closed orientable surfaces (Cf. next subsection). The $1$-formality and formality (implying $1$-formality) properties imply filtered-formality. This gives a wide class of examples.
\end{remark}
\subsection{Applications to orbit configuration spaces of surfaces}\label{S42}
\begin{proposition}
For $S\neq S^2$ the Malcev Lie algebra of $C_n^G(S)$ determines entirely the minimal and $1$-minimal model: the cdga $A_{C_n^G(S)}$ of proposition \ref{prop CCE model} is a minimal and $1$-minimal model. The cdga $A_{C_n^G(S^2)}$ is a $1$-minimal model of $C_n^G(S^2)$.
\end{proposition}
\begin{proof}
We have shown in subsection \ref{S31} that $C_n^G(S)$ is a rational $K(\pi,1)$ space if and only if $S\neq S^2$. Hence, the proposition follows from proposition \ref{prop CCE model}.
\end{proof}
We have proved the part concerning Malcev Lie algebras of $(3)$.\\\\
Let $G$ be a finite subgroup of $\PGL(\C^2)$ acting (naturally) by homography on $\PP^1\simeq S^2$ (the complex projective line) and $\PP^1_*$ be the set of points with trivial stabilizers with respect to the action of $G$. We have shown in \cite{MM}, that the Malcev Lie algebra of $\pi_1 C_n^G(\PP^1_*)$ is the degree completion of a graded Lie algebra $\mathfrak{p}_n(G)$ defined by generators and relations. We note that $\PP_*^1\neq \PP^1$ if and only if $G\neq \{1\}$ (see for instance the classification of the spaces $C_n^H(S^2)$ in subsection \ref{S21})
\begin{proposition}
Take $G$ and $\PP^1_*$ as in the previous paragraph. The cdga $C_{\mathrm{gr}}^*(\mathfrak{p}_n(G))$ is the $1$-minimal model of $C_n^G(\PP^1_*)$. It is also the minimal model if $G\neq \{1\}$.
\end{proposition}
\begin{proof}
The proposition is an application of proposition \ref{prop CCE model}, since $\PP^1_*\neq \PP^1$ if and only if $G\neq \{1\}$ and hence $C_n^G(\PP^1_*)$ is a rational $K(\pi,1)$ if and only if $G\neq \{1\}$ by results of section \ref{S3}.
\end{proof}
We have proved main result $(6)$ of the introduction.\\\\
For $S=\C$, one can prove that $L_{C_n(\C)}$ is the degree completion of the Kohno-Drinfeld Lie algebra $\mathfrak{t}_n$ and hence proposition \ref{prop CCE model} also applies to the rational $K(\pi,1)$ space $C_n(\C)$. In general, for $\bar{S}=S^2$, $C_n^G(S)$ is formal, for $G$ orientation preserving \cite{MM3}, and proposition \ref{prop CCE model} applies (Cf. remark \ref{rmk CCE}). Malcev Lie algebras associated to configuration spaces (the case $G=\{1\}$) of orientable closed surfaces are degree completions of some known graded Lie algebras (Cf. \cite{Bezr}, \cite{BE3}, for genus $g\geq 1$), and in \cite{Bezr} $C_{\mathrm{gr}}(-)$ is used to descibe the minimal model of these configuration spaces. In \cite{CalMart} (section 5), it is proved that the Malcev Lie algebra of $\pi_1C_n^{E_p}(E)$, where $E_p$ is the $p$-torsion subgroup of and elliptic curve $E$ acting by translation on $E$, is the degree completion of some graded Lie algebra definied by generators and relations and hence proposition \ref{prop CCE model} also applies to the graded Lie algebras.
\subsection{On Malcev Lie algebras of orbit configuration spaces of surfaces}\label{S43}
As examples we give the Malcev Lie algebras of $S=C_1^G(S)$ for $S$ as in the general case (i.e. $S$ cofinite in a closed orientable surface):
\begin{proposition}\label{prop Malcev el}
For $Z$ a tuple, we denote by $\mathrm{Lie}(Z)$ the (graded) free Lie algebra on $Z$ where the components of $Z$ are of degree $1$, and with the convention $L(Z)=0$ if $Z=\emptyset$.
\begin{itemize}
\item[1)] If $S$ is not compact, then $L_S$ is the degree completion of $\mathrm{Lie}(x_1,\dots,x_l)$, where $l=\mathrm{dim}(H^1(S,\Q))-1$.
\item[2)] If $S$ is compact of genus $g\geq 0$, then $L_S$ is the degree completion of $$\mathrm{Lie}(a_1,b_1,\dots,a_g,b_g)/I,$$ where $g$ is the genus of $S$, $I$ is the ideal generated by $[a_1,b_1]+\cdots +[a_g,b_g]$, and $a_i,b_i$ have degree $1$.
\end{itemize}
\end{proposition}
Recall that for $k<n$ $C_{n-k}^G(S_{g,k})$ is the fiber of $\pi{n,k}:C_{n}^G(S)\to C_k^G(S)$.
\begin{proposition}
Take $n>k\geq 1$. 
\begin{itemize}
\item[1)] There is an exact sequence of Malcev Lie algebras:
$$ 0 \to \mathrm{L}_{C_{n-k}^G(S_{G,k})} \to \mathrm{L}_{C_{n}^G(S)} \to \mathrm{L}_{C_{k}^G(S)} \to 0,$$
in the following cases:
 \begin{itemize} 
\item $S$ not homeomorphic to $S^2$, 
\item$S\simeq S^2$, $G=\{1\}$ and $k\geq 3$, 
\item $S\simeq S^2$, $G\neq \{1\}$ and $k\geq 2$.
\end{itemize}
\item[2)] If $S\simeq S^2$, then $L_{C_{1}^G(S)}=L_{C_{2}^G(S)}=L_{C_{3}(S)}=0$.
\end{itemize}
\end{proposition}
\begin{proof}
A morphism $(\Lambda V,d_V) \to (\Lambda W,d_W)$ of minimal Sullivan algebras induces naturally (by restricting to degree $1$ then dualizing) a morphism of Lie algebras $L_{(\Lambda W,d_W)} \to L_{\Lambda (V,d_V)}$. We have seen (section \ref{S3}) that in the cases mentioned in $(1)$, that $C_{n}^G(S)$ have a minimal model $A_{n,k}$ fitting in a sequence of minimal Sullivan algebras:
$$\Lambda V_{C_k^G(S)}\to A_{n,k}=\Lambda V_{C_k^G(S)}\otimes \Lambda V_{C_{n-k}^G(S_{G,k})} \to \Lambda V_{C_{n-k}^G(S_{G,k})},$$
where the map on the left is $\mathrm{id} \otimes 1$ and the one on the right is obtained by applying the augmentation map of $\Lambda V_{C_k^G(S)}$. The above sequence induces an exact sequence: $$0\to V_{C_k^G(S)}\to V_{C_k^G(S)}\oplus V_{C_{n-k}^G(S_{G,k})} \to V_{C_{n-k}^G(S_{G,k})}\to 0,$$ of graded vector spaces. This with the observation made at the beginning of the proof, shows that we have an exact sequence of Lie algebras: $$0\to L_{\Lambda V_{C_{n-k}^G(S_{G,k})}}\to L_{\Lambda V_{C_{n}^G(S)}}\to L_{C_k^G(S)} \to 0,$$
where we omit the differentials in the subscripts. We obtain the exact sequence of the proposition by applying the isomorphism $L_{\Lambda V_X}\simeq L_X$ (Cf. proposition \ref{prop HLie Mal}). We prove $(2)$. We have $\pi_\psi^*(C_1^G(S^2))=V_{S^2}$,  $\pi_\psi^*(C_3(S^2))=\pi_\psi^*(C_2^{\Z/2\Z}(S^2))=(\Q)_3$  by $(2)$ of proposition \ref{prop models S2}, and $\pi_\psi^*(C_2(S^2))=V_{S^2}\oplus V_{\R^2}$ by $(1.a)$ of the same proposition. The vector space $V_{\R^2}$ is null and $V_{S^2}$ (as for $(\Q)_3$) has no degree $1$ part (as seen previously). This proves the second claim of the proposition. We have proved the proposition.
\end{proof}
By a complete free Lie algebra on $l\geq 0$ generators, we mean the degree completion of the free Lie algebra on $l$ generators $x_1,\dots,x_l$, where $x_i$ is of degree $1$ (for $l=0$ the Lie algebra is $0$).
\begin{corollary}
The Malcev Lie algbera of $C_{n}^G(S)$ is obtained by iterated extensions by free Lie algebras of $L_S$, and $L_S$ is one of the Lie algebras in \ref{prop Malcev el}.
\end{corollary}
We have proved main result $(5)$. 
\section{Koszulity in genus zero}\label{S5}
In this section, we assume that $\bar{S}=S^2$ and that $G$ is a finite group acting freely by orientation preserving homeomorphisms on $S$. We have proved, in \cite{MM3}, that $C_n^G(S)$ is formal and computed its cohomology ring. It follows from the classification of orbit configuration spaces on $S^2$, in subsection \ref{S22}, that if $S=S^2$ then $G=\{1\}$. In the other cases ($S\neq S^2$), $C_n^G(S)$ is a rational $K(\pi,1)$ by results of section \ref{S3}. As a consequence of theorem \ref{Koszul KP1} (taken from \cite{PapYuz}), we get:
\begin{proposition}
The cohomology ring of $C_n^G(S)$, with $\bar{S}=S^2$ and $G$ orientation preserving, is Koszul if and only if $(G,S)\neq (\{1\},S^2)$.
\end{proposition}
We note that in \cite{MM3}, relations and generators defining the cohomology ring of $C_n^G(S)$ (as above) are given. A basis is also given. One can prove that the basis is a PBW basis in the sense of \cite{Prid} and hence obtains another proof of Koszulity for $(G,S)\neq (\{1\},S^2)$  by theorem 5.3 of \cite{Prid}. \\\\
As another application of the rational $K(\pi,1)$ property and formality we get (by corollary 5.3 of \cite{PapYuz}) an LCS formula for $(G,S)\neq (\{1\},S^2)$:
$$P_{C_n^G(S)}(-t)=\prod_{i\geq 1} (1-t^i)^{\phi_i(\pi_1C_n^G(S))},$$
where $P_{C_n^G(S)}$ is the Poincaré polynomial of $C_n^G(S)$ and $\phi_i(\pi_1C_n^G(S))$ is the rank of the quotient $\Gamma_i \pi_1C_n^G(S)/\Gamma_{i+1} \pi_1C_n^G(S)$, with $\{\Gamma_i \pi_1 C_n^G(S)\}_{i\geq 1}$ the lower central series of $\pi_1C_n^G(S)$. This was established using other arguments in \cite{MM3} and the Poincaré polynomial wase computed. A similar formula for $(G,S)= (\{1\},S^2)$ is also given. If $C_n^G(S)$ correspond to one the spaces $C_n^G(\PP^1_*)$ mentioned in subsection \ref{S43}, then the constants $\phi_i(\pi_1C_n^G(\PP^1_*))$ correspond to the dimension of the degree $i$ component of $\mathfrak{p}_n(C_n^G(\PP^1_*))$ mentioned in subsection \ref{S43}.
 
\section{Appendix}
Here we prove proposition \ref{prop HLie Mal} of section \ref{S41}. For that we use a correspondence between the constructions of "Malcev completions" from \cite{Q1} and \cite{FelHar2} (Cf. remark \ref{rmk correspondance} of \cite{FelHar2}).\\\\
For $A$ a group or a Lie algebra, we mean by a filtration a sequence of subgroups or Lie subalgebras $A=F_1 A \supset F_2 A \cdots \supset \cdots$ such that $F_iA\star F_j A \subset F_{i+j} A$, where $F_ iA \star F_j$ is either the subgroup generated by the commutators $aba^{-1}b^{-1}$ with $a \in F_i A$ and $b\in F_j A$, either the Lie subalgebra $[F_iA, F_j A]$. A Lie algebra or a group is called filtered if it is equipped with a filtration. For instance, the lower central series $\{\Gamma_i A\}_i$ given by $\Gamma_1 A=A$ and $\Gamma_{i+1}A=\Gamma_{i}A\star A$, with $\star$ as before is a filtration. Moreover, it is the thinnest filtration, i.e. if $\{F_i A\}_i$ is a filtration then $\Gamma_i A \subset F_i A$ for $ii\geq 1$. If $A$ and $A'$ are filtered (groups or Lie algebras) a filtered morphism $f:A\to A'$ is a morphism (of groups or Lie algebras) such that $f(F_i A) \subset f(F_iA')$. A filtered isomorphism is a filtered morphism admitting a filtered inverse. The associated graded of a filtered $A$ with filtration $\{F_iA\}_i$ is the graded abelian group or vector space $\mathrm{gr} A=\oplus_{i\geq 1} F_i A/F_{i+1}A$. If $f:A\to A'$ is a filtered morphism then $f$ induces a morphism $\mathrm{gr}(f): \mathrm{gr}A \to \mathrm{gr} A'$. \\\\
We recall some elements from \cite{Q1}. The author introduces the category of Malcev Lie algebras ($\mathrm{MLA}$) and Malcev groups ($\mathrm{MGp}$). The objects of these categories are filtered Lie algebras or filtered groups, with some extra conditions, and morphisms are filtered morphisms. He also introduces the category of complete Hopf algebras ($\mathrm{CHA}$), and two functors $\mathcal{G}_F:\mathrm{CHA} \to \mathrm{MGp}$ and $\mathcal{P}_F:\mathrm{CHA} \to \mathrm{MLA}$. The first functor associates to an object of CHA the group of group-like elements equipped with a given filtration, while the second associate the Lie algebra of primitive elements equipped with a given filtration. Let $G$ be a group, $\Q[G]$ the group algebra and $I$ the augmentation ideal of $\Q[G]$. The $I$-adic completion $\Q[G]\:\hat{}:=\varprojlim \Q[G]/I^i$ of $\Q[G]$ is a complete Hopf algebra. The Malcev Lie algebra of $G$ is the object $\mathcal{P}_F( \Q[G]\:\hat{})$. We will denote by $L(G)$ and $\{F_i L(G)\}_i$ the underlying algebra and filtration of $\mathcal{P}_F( \Q[G]\:\hat{})$. The Malcev group of $G$ is the object $\mathcal{P}_F( \Q[G]\:\hat{})$. We will use the notation $G_\Q$ for the group and $\{F_i G_\Q\}_i$ for the filtration. By construction, we have a natural group morphism $j_G: G\to G_\Q$, corresponding to the composite $G\to\Q[G]\to \Q[G]\:\hat{}$. If $G$ is equipped with the lower central series and $G_\Q$ with its "Malcev" filtration, then $j_G$ is a filtered morphism. The natural maps:
$$ L(G)\to \varprojlim L(G)/F_i L(G) \quad \text{and} \quad G_\Q\to \varprojlim G_Q/F_i G_\Q,$$
are isomorphisms and the natural map $\mathrm{gr} G\otimes \Q \to \mathrm{gr}G_\Q$ is an isomorphism. We note that $\cap_ i F_i G_\Q=\{1\}$.
\begin{lemma}\label{lem grf}
Let $f: G\to G'$ be a map of filtered groups $G,G'$ with filtrations $\{F_iG\}_i,\{F_i G'\}_i$. Assume that $F_k G'=\{1\}$ for a given $k$ and that $\cap_ i F_i G=\{1\}$. The map $f$ is an isomorphism of filtered groups if and only if $\mathrm{gr}(f)$ is an isomorphism.
\end{lemma}
We leave the proof of the lemma for the reader. The natural map $j_G$ induces a morphism $j_{G,i}:G/\Gamma_i G \to G_\Q/F_i G_\Q$, because $\Gamma_i G_\Q\subset F_iG_i\Q$.
\begin{proposition}
The filtration $\{F_i (G/\Gamma_iG)_\Q\}_{i\geq1}$ and the one induced by $\{F_i G_\Q\}_{i\geq 1}$ on $G_\Q/F_i G_\Q $ are the corresponding lower central series. There exists a group isomorphism $f_{G,i}: G_\Q/F_i G_\Q \to (G/\Gamma_iG)_\Q$, making the following diagram commutative:
\[\begin{tikzcd}
G/\Gamma_i G\arrow{r}{j_{G,i}} \arrow{rd}[swap]{j_{G/\Gamma_iG}}&G_\Q/F_i G_\Q \arrow{d}{f_{G,i}}\\
& (G/\Gamma_iG)_\Q
\end{tikzcd}\]
In particular, the map $\mathrm{gr} G/ \Gamma_i G \otimes \Q \to \mathrm{gr} G_\Q/F_i G_\Q$, induced by $j_{G,i}$, is an isomorphism.
\end{proposition}
\begin{proof}
By construction the quotient map $G\to G/\Gamma_iG$ gives a commutative diagram of filtered group morphisms:
\[ \begin{tikzcd}
G\arrow[d] \arrow{r}{j_G} & G_\Q \arrow{d}{f_i} \\
G/\Gamma_i G\arrow{r}{j_{G/\Gamma_i G}} & (G/\Gamma_i G)_\Q
\end{tikzcd}\]
where the groups on the left are equipped with their lower central series filtration, the groups on the right are equipped with there "Malcev filtration" and $G\to G/\Gamma_i G$ is the quotient map. Using the isomorphism $\mathrm{gr}( G/\Gamma_i G )\otimes \Q \to \mathrm{gr} (G/\Gamma_i G )_\Q$ on deduces that $\mathrm{gr} (G/\Gamma_i G )_\Q$ is null in degree $l\geq i$. Hence, $F_i(G/\Gamma_i G )_\Q=\{1\}$, since $\cap_{k\geq 0} F_k K_\Q=\{1\}$ for any $K$. This proves that $F_i G_\Q$ lies in the kernel of $f_i$ and that (\cite{Q1}, proposition 3.5, p. 276) $\{F_i(G/\Gamma_i G )_\Q\}_{i\geq 1}$ is the lower central series of $(G/\Gamma_i G)_\Q$. Hence, we have a commutative diagam:
\[ \begin{tikzcd}
G \arrow[d] \arrow[r] & G_\Q/F_i G_\Q \arrow{d}{\bar{f}_i} \\
G/\Gamma_i G \arrow[r] \arrow[ur] & (G/\Gamma_i G)_\Q
\end{tikzcd}\]
where $\bar{f}_i$ is the map of filtered groups (with filtrations as in the beginning of the proposition) induced by $f_i$. To prove our claim, we only need to prove that $\mathrm{gr}\bar{f_i}$ is an isomorphism (Cf. lemma \ref{lem grf}). One proves this by comparing the associated graded of $G, G_\Q/F_i G_\Q, G/\Gamma_i G$ and $(G/\Gamma_i G)_\Q$ using definitions and the isomorphisms $\mathrm{gr} K\otimes \Q\to \mathrm{gr} K_\Q$.
\end{proof}

\begin{definition}\cite{FelHar2}
A Malcev $\Q$-completion of a group $G$ is a group morphism $\varphi: G\to G_c$, where:
\begin{itemize}
\item[1)] the natural map $G_c\to \varprojlim G_c/ \Gamma_i {G_c}$ is an isomorphism,
\item[2)] for $i\geq 1$ the $\Z$-module structure on $\mathrm{gr} G_c$ extends to a $\Q$-vector space structure,
\item[3)] the map $\mathrm{gr}G\otimes \Q\to \gr G_c$ induced by $\varphi$ is an isomorphism.
\end{itemize}
\end{definition}
 
\begin{proposition}\label{prop comp}
If $\varphi: G\to G_c$ is a Malcev $\Q$-completion of a group $G$, then we have a group isomorphism $G_\Q\to G_c$ making the diagram:
\[ \begin{tikzcd}
\: & G_\Q \arrow{d} \\
G \arrow{r}[swap]{\varphi}\arrow{ur}{j_G}&G_c
\end{tikzcd}\]
commutative, and the filtration $\{F_i G_\Q\}_{i\geq 1}$ is the lower central series of $G_\Q$.
\end{proposition}
\begin{proof}
The group $G_c/\Gamma_i G_c$ is nilpotent and the $\Z$-module structure of $\mathrm{gr} (G_c/\Gamma_i G_c)$ extends to a $\Q$-vector space structure. Hence, $ G_c/\Gamma_i G_c$ is uniquely divisble (\cite{Q1}, $(ii)$ of corollary 3.7, p. 277-278).
The map $j_{G/\Gamma_i G}: \Gamma_iG\to (G/ \Gamma_iG)_\Q$ is universal with respect to morphisms from $G/ \Gamma_iG$ into nilpotent uniquely divisible groups (\cite{Q1}, corollary 3.8, p. 278). Hence, by the previous proposition, we have a commutative diagram of filtered group morphisms:
\[ \begin{tikzcd}
\: & G_\Q/F_i G_\Q \arrow{d}{f_i} \\
G/\Gamma_i G\arrow[r] \arrow[ur] & G_c/\Gamma_i G_c
\end{tikzcd}\]
Since $\mathrm{gr}G\otimes \Q \to \mathrm{gr}G_c$ is an isomorphism, the similar map obtained from $G/\Gamma_i G \to G_c/\Gamma_i G_c$ is also an isomorphism. Moreover, we have proved in the previous proposition that $\mathrm{gr} G/ \Gamma_i G\otimes \Q \to \mathrm{gr} G_\Q/F_i G_\Q$ is an isomorphism. Hence, $\mathrm{gr}(f_i)$ is an isomorphism. This proves that $f_i:G_\Q/F_i G_\Q\to G_c \Gamma_i G_c$ is an isomorphism (Cf. lemma \ref{lem grf}). The isomorphisms $f_i:G_\Q/F_i G_\Q\to G_c \Gamma_i G_c$ for different $i$'s are compatible, since they are obtained using a universal map. This gives an isomorphism $f: G_\Q\to G_c$, since the natural maps $G_\Q\to \varprojlim G_\Q/F_i G_\Q$ and $G_c\to \varprojlim G_c/\Gamma_i G_c$ are isomorphisms. The assertion on the filtration follows, since: $F_iG_\Q=\mathrm{Ker}(G_\Q\to G_\Q/F_iG_\Q)$, $\Gamma_i G_c=\mathrm{Ker}( G_c\to G_c/\Gamma_iG_c)$ and $f$ is obtained out of the isomorphisms $f_i$.
\end{proof}
\begin{remark}\label{rmk correspondance}
The above proposition establishes the correspondence between the construction of $G_\Q$ from \cite{Q1} and the Malcev $\Q$-completion of $G$ in the sense of \cite{FelHar2}, for $G$ admitting a Malcev $\Q$-completion. Malcev $\Q$-completions are constructed in \cite{FelHar2} (theorem 7.5, p. 214) for $G$ admitting an abelianization of finite rank using classifying spaces (see also proposition \ref{Mal comp FHT}).
\end{remark}
For $(\Lambda V,d)$ a minimal Sullivan algebra with $H^1(\Lambda V)$ finite dimensional, the natural map:
\begin{equation}\label{Malcev LLX}
{L_{\Lambda V}}\to \varprojlim L_{\Lambda V}/ \Gamma_i L_{\Lambda V},\end{equation}
where $L_{\Lambda V}$ is the fundamental Lie algebra of $(\Lambda V,d)$ (definition given in \ref{S41}), is an isomorphism, and $L_{\Lambda V}/ \Gamma_2 L_{\Lambda V}$ is isomorphic to the finite dimensional vector space $H^1(\Lambda V)$ (\cite{FelHar2},§2.3). Now consider the enveloping algebra $UL_{\Lambda V}$ of ${L_{\Lambda V}}$, denote by $I_{L_{\Lambda V}}$ its augmentation ideal and set $\widehat{U}{L_{\Lambda V}}= \varprojlim UL_{L_{\Lambda V}}/I^i_{L_{\Lambda V}}$. The algebra $\widehat{U}{L_{\Lambda V}}$ is a complete Hopf algebra. One associates to $L_{\Lambda V}$ the group $G_{L_{\Lambda V}}$ (\cite{FelHar2}, §2.4, definition p. 59) of group-like elements of $\widehat{U}L_{\Lambda V}$. Moreover, the natural map $L_{\Lambda V} \to \widehat{U}{L_{\Lambda V}}$ induces an isomorphism between $L_{\Lambda V}$ and the Lie algebra of primitive elements of $\widehat{U}{L_{\Lambda V}}$ (\cite{FelHar2}, corollary 2.3, p. 59).\\\\
For $\Lambda V$ a minimal Sullivan algebra the fundamental group of $\Lambda V$ is $\pi_1(\Lambda V)=G_{L_{\Lambda V}}$. For $X$ a path connected space, the quasi-isomorphism $\Lambda V_X \to A_{PL}(X)$ induces a group morphism $\phi_X:\pi_1(X)\to \pi_1(\Lambda V_X)$.
\begin{proposition}\cite{FelHar2}\label{Mal comp FHT}
If $X$ is path connected and $H^1(X,\Q)$ is finite dimensional then $\varphi_X : \pi_1X\to \pi_1(\Lambda V_X)$ is a Malcev $\Q$-completion.
\end{proposition}
This is corollary 7.4 (p.214) of the corresponding reference.\\\\
\textbf{Proof of proposition \ref{prop HLie Mal}}\: \:We have a path connected topological space $X$ with $H^1(X,\Q)$ finite dimensional. We want to show that the Malcev Lie algebra $L_X$ of $\pi_1X$ is isomorphic to the fundamental Lie alebra $L_{\Lambda V_X}$, and that the Malcev filtration on $L_X$ is the lower central series. The Lie algebras and groups $L_{\Lambda V_X},\pi_1(\Lambda V_X), L_X$ and $(\pi_1X)_\Q$ are the underlying Lie algebras and groups of $\mathcal{P}_F(\hat{U}L_{\Lambda V_X}),\mathcal{G}_F(\hat{U}L_{\Lambda V_X}),\mathcal{P}_F(\Q[\pi_1X]\:\hat{})$ and $\mathcal{G}_F(\Q[\pi_1X]\:\hat{})$. Moreover, by propositions \ref{prop comp} and \ref{Mal comp FHT}, $(\pi_1X)_\Q \simeq \pi_1(\Lambda V_X)$ and the filtration on $\mathcal{G}_F(\Q[\pi_1X]\:\hat{})$ is the lower central series. Since the functors $\mathcal{P}_F$ and $\mathcal{G}_F$ are category equivalences (\cite{Q1}, theorem 3.3, p. 275), in order to prove the proposition it suffices to prove that filtrations of $\mathcal{G}_F(\hat{U}L_{\Lambda V_X})$ and $\mathcal{P}_F(\Q[\pi_1X]\:\hat{})$ are the corresponding lower central series.
\begin{proposition}\label{prop CHA}
Let $H$ be a complete Hopf algebra and $I_H$ its augmentation ideal. If $I_H/I_H^2$ is finite dimensional, then the filtrations of $\mathcal{P}_F(H)$ and $\mathcal{G}_F(H)$ are the corresponding lower central series.
\end{proposition}
\begin{proof}
Let $x_1,\dots,x_m$ be elements of $H$ lifting a basis of $I_H/I_H^2$ and let $H':=\Q\langle \langle \{x_l\}_l \rangle \rangle$ be the complete Hopf algebra of formal series, with coefficients in $\Q$, in the non commuting variables $\{x_l\}_l$. One has a unique complete Hopf algebra morphism $H' \to H $ given by $x_l \mapsto x_l$. The morphism maps $I_{H'}$ onto $I_H$. By definition:
$$ F_i \mathcal{P}_F(H)= I_H^i \cap \mathcal{P}_F(H) \quad, \quad F_i \mathcal{G}_F(H)=(1+ I_H^i) \cap \mathcal{G}_F(H),$$
where $ \{F_i \mathcal{P}_F(H)\}$ and $\{F_i \mathcal{G}_F(H)\}_i $ are the filtrations of $\mathcal{P}_F(H)$ and $\mathcal{G}_F(H)$. Hence, we only need to prove the proposition for $H'$ and we will do so. Denote by $L$ the free Lie algebra on $\{x_l\}_l$ and assume that $x_l$ has degree $1$. We have $L=\oplus_{i\geq 1} L_i$, where $L_i$ is the degree $i$ component of $L$. The underlying Lie algebra of $\mathcal{P}_F(H')$ correspond the degree completion $\hat{L}=\prod_{i\geq 1} L_i$ of $L$ (\cite{Q1}, example 2.11, p. 272), and the filtration on $\mathcal{P}_F(H')$ is given by $F_i \hat{L}=\prod_{j\geq i} L_j \subset \hat{L}$. By lemma \ref{lem filgrL}, $F_i \hat{L}=\Gamma_i \hat{L}$. This proves the part concerning $\mathcal{P}_F(H)$ of the proposition. We now consider $\mathcal{G}_F(H')$. We have an isomorphism of complete Hopf algebras $H'\simeq \Q[K]\:\hat{}$, where $K$ is the free groups on $\{x_l\}_l$ (\cite{Q1}, example 2.11, p. 272). Hence, $\mathcal{G}_F(H')\simeq \mathcal{G}_F(\Q[K]\:\hat{})$ (as filtered groups). By propositions \ref{prop comp} and \ref{Mal comp FHT} (for $X$ a bouquet of $m$ circles), the filtration of $\mathcal{G}_F(\Q[K]\:\hat{})$ is the lower central series. This proves that the filtration of $\mathcal{G}_F(H')$ is the lower central series. We have proved the proposition.
\end{proof}
We note that for $H$ a complete Hopf algebra we have isomorphisms: $\mathrm{gr}_1 \mathcal{P}_F(H) \simeq I_H/I_H^2 \simeq \mathrm{gr}_1\mathcal{G}_F(H)$ (\cite{Q1}, example 3.2, p. 275).
\begin{proposition}
The filtration of $\mathcal{G}_F(\hat{U}L_{\Lambda V_X})$ and $\mathcal{P}_F(\Q[\pi_1X]\:\hat{})$ are the corresponding lower central series.
\end{proposition}
\begin{proof}
In light of the previous proposition, we only prove that $I_H/I_H^2$ is finite dimensional for $H=\Q[\pi_1X]\:\hat{}$ and $H=\hat{U}L_{\Lambda V_X}$. We have isomorphisms $H^1(X,\Q) \simeq (\mathrm{gr}_1 \pi_1X) \otimes \Q \simeq \mathrm{gr}_1 (\pi_1 X)_\Q \simeq I_{\Q[\pi_1X]\:\hat{}}/I_{\Q[\pi_1X]\:\hat{}}^2$. This prove that $I_{\Q[\pi_1X]\:\hat{}}/I_{\Q[\pi_1X]\:\hat{}}^2$ is finite dimensional, since $H^1(X,\Q)$ is finite dimensional. The vector space $I_{\hat{U}L_{\Lambda V_X}}/I_{\hat{U}L_{\Lambda V_X}}^2 \simeq \mathrm{gr}_1 \mathcal{P}_F(\hat{U}L_{\Lambda V_X}) $ is isomorphic to the quotient of $L_{\Lambda V_X} $ by the second term of the filtration. The second term of the filtration contains $\Gamma_2 L_\Lambda V_X$ and hence it suffice to prove that $L_{\Lambda V_X}/\Gamma_2 L_{\Lambda V_X}$ is finite dimensional. Since $H^1(X,\Q)$ is finite dimensional, we have (as seen previously) $L_{\Lambda V_X}/\Gamma_2 L_{\Lambda V_X} \simeq H^1(X,\Q)$ and hence $L_{\Lambda V_X}/\Gamma_2 L_{\Lambda V_X}$ is finite dimensional. We have proved the proposition.
\end{proof}
This proves proposition \ref{prop HLie Mal} of subsection \ref{S41}, as we have discussed in the beginning of the proof.

\bibliographystyle{alpha}
\bibliography{Biblio}{}

\begin{thebibliography}{GGSX15}

\bibitem[Bez94]{Bezr}
Roman Bezrukavnikov.
\newblock {Koszul DG-algebras arising from configuration spaces}.
\newblock {\em Geom. Funct. Anal.}, 4(2):119--135, 1994.

\bibitem[BG76]{BousGug}
A.~K. Bousfield and V.~K. A.~M. Gugenheim.
\newblock On {${\rm PL}$} de {R}ham theory and rational homotopy type.
\newblock {\em Mem. Amer. Math. Soc.}, 8(179):ix+94, 1976.

\bibitem[BGS96]{BeilKoszul}
Alexander Beilinson, Victor Ginzburg, and Wolfgang Soergel.
\newblock Koszul duality patterns in representation theory.
\newblock {\em J. Amer. Math. Soc.}, 9(2):473--527, 1996.

\bibitem[BH16]{BibbyHil}
Christin Bibby and Justin Hilburn.
\newblock Quadratic-linear duality and rational homotopy theory of chordal
  arrangements.
\newblock {\em Algebr. Geom. Topol.}, 16(5):2637--2661, 2016.

\bibitem[BK72]{BousKan}
A.~K. Bousfield and D.~M. Kan.
\newblock {\em Homotopy limits, completions and localizations}.
\newblock Lecture Notes in Mathematics, Vol. 304. Springer-Verlag, Berlin-New
  York, 1972.

\bibitem[Cas16]{casto}
Kevin Casto.
\newblock $\mathrm{FI}_g$-modules, orbit configuration spaces, and complex
  reflection groups, 2016.

\bibitem[CG19]{CalMart}
Damien Calaque and Martin Gonzalez.
\newblock On the universal ellipsitomic kzb connection, 2019.

\bibitem[CIW19]{IdCamWill}
Ricardo Campos, N.~Idrissi, and T.~Willwacher.
\newblock Configuration spaces of surfaces.
\newblock {\em arXiv: Quantum Algebra}, 2019.

\bibitem[CKX09]{CKX}
Fred Cohen, Toshitake Kohno, and Miguel Xicot{\'{e}}ncatl.
\newblock {Orbit configuration spaces associated to discrete subgroups of
  $\mathrm{PSL}_2(\mathbb{R})$}.
\newblock {\em J. Pure Appl. Algebra}, 213(12):2289--2300, 2009.

\bibitem[Coh01]{DCoh}
Daniel Cohen.
\newblock {Monodromy of fiber-type arrangements and orbit configuration
  spaces}.
\newblock {\em Forum Math.}, 13(4):505--530, 2001.

\bibitem[CX02]{XicoCohGauss}
F.~R. Cohen and M.~A. Xicot\'{e}ncatl.
\newblock On orbit configuration spaces associated to the {G}aussian integers:
  homotopy and homology groups.
\newblock volume 118, pages 17--29. 2002.
\newblock Arrangements in Boston: a Conference on Hyperplane Arrangements
  (1999).

\bibitem[Enr07]{BE}
Benjamin Enriquez.
\newblock {Quasi-reflection algebras and cyclotomic associators}.
\newblock {\em Selecta Math. (N.S.)}, 13(3):391--463, 2007.

\bibitem[Enr14]{BE3}
Benjamin Enriquez.
\newblock {Flat connections on configuration spaces and braid groups of
  surfaces}.
\newblock {\em Adv. Math.}, 252:204--226, 2014.

\bibitem[Fal88]{FMmin}
Michael Falk.
\newblock The minimal model of the complement of an arrangement of hyperplanes.
\newblock {\em Trans. Amer. Math. Soc.}, 309(2):543--556, 1988.

\bibitem[FHT01]{FelHar}
Yves F\'{e}lix, Stephen Halperin, and Jean-Claude Thomas.
\newblock {\em Rational homotopy theory}, volume 205 of {\em Graduate Texts in
  Mathematics}.
\newblock Springer-Verlag, New York, 2001.

\bibitem[FHT15]{FelHar2}
Yves F\'{e}lix, Steve Halperin, and Jean-Claude Thomas.
\newblock {\em Rational homotopy theory. {II}}.
\newblock World Scientific Publishing Co. Pte. Ltd., Hackensack, NJ, 2015.

\bibitem[FM94]{FM}
William Fulton and Robert MacPherson.
\newblock {A compactification of configuration spaces}.
\newblock {\em Ann. of Math. (2)}, 139(1):183--225, 1994.

\bibitem[FZ02]{ziegAnt}
Eva~Maria Feichtner and Günter~M. Ziegler.
\newblock On orbit configuration spaces of spheres.
\newblock {\em Topology and its Applications}, 118(1):85 -- 102, 2002.
\newblock Arrangements in Boston: A Conference on Hyperplane Arrangements.

\bibitem[GG04]{GuaSP}
D.~L. Gon\c{c}alves and J.~Guaschi.
\newblock On the structure of surface pure braid groups.
\newblock {\em J. Pure Appl. Algebra}, 186(2):187--218, 2004.
\newblock Corrected reprint of: ``On the structure of surface pure braid
  groups'' [J. Pure Appl. Algebra {{\bf{1}}82} (2003), no. 1, 33--64;
  MR1977999].

\bibitem[GGSX15]{XicoAnt2}
Jes\'{u}s Gonz\'{a}lez, Aldo Guzm\'{a}n-S\'{a}enz, and Miguel Xicot\'{e}ncatl.
\newblock The cohomology ring away from 2 of configuration spaces on real
  projective spaces.
\newblock {\em Topology Appl.}, 194:317--348, 2015.

\bibitem[GM13]{MGRH}
Phillip Griffiths and John Morgan.
\newblock {\em Rational homotopy theory and differential forms}, volume~16 of
  {\em Progress in Mathematics}.
\newblock Springer, New York, second edition, 2013.

\bibitem[Hal78]{HarRF}
Stephen Halperin.
\newblock Rational fibrations, minimal models, and fibrings of homogeneous
  spaces.
\newblock {\em Trans. Amer. Math. Soc.}, 244:199--224, 1978.

\bibitem[Hal83]{HarLMM}
S.~Halperin.
\newblock Lectures on minimal models.
\newblock {\em M\'{e}m. Soc. Math. France (N.S.)}, (9-10):261, 1983.

\bibitem[Ive86]{Iver}
Birger Iversen.
\newblock {\em Cohomology of sheaves}.
\newblock Universitext. Springer-Verlag, Berlin, 1986.

\bibitem[Kri94]{Kriz}
Igor Kriz.
\newblock {On the rational homotopy type of configuration spaces}.
\newblock {\em Ann. of Math. (2)}, 139(2):227--237, 1994.

\bibitem[Maa19]{MM}
Mohamad Maassarani.
\newblock {Sur certains espaces de configuration associ{\'{e}}s aux
  sous-groupes finis de PSL2(C)}.
\newblock {\em Bulletin de la Soci{\'{e}}t{\'{e}} math{\'{e}}matique de
  France}, 2019.

\bibitem[Maa20]{MM3}
Mohamad Maassarani.
\newblock Algebraic invariants of orbit configuration spaces in genus zero
  associated to finite groups, 2020.

\bibitem[Pri70]{Prid}
Stewart~B. Priddy.
\newblock Koszul resolutions.
\newblock {\em Trans. Amer. Math. Soc.}, 152:39--60, 1970.

\bibitem[PS04]{PapSucChenLie}
Stefan Papadima and Alexander~I. Suciu.
\newblock Chen {L}ie algebras.
\newblock {\em Int. Math. Res. Not.}, (21):1057--1086, 2004.

\bibitem[PY99]{PapYuz}
Stefan Papadima and Sergey Yuzvinsky.
\newblock On rational {$K[\pi,1]$} spaces and {K}oszul algebras.
\newblock {\em J. Pure Appl. Algebra}, 144(2):157--167, 1999.

\bibitem[Qui69]{Q1}
Daniel Quillen.
\newblock {Rational homotopy theory}.
\newblock {\em Ann. of Math. (2)}, 90:205--295, 1969.

\bibitem[SW19]{SucWan}
Alexander~I. Suciu and He~Wang.
\newblock Formality properties of finitely generated groups and lie algebras.
\newblock {\em Forum Mathematicum}, 31(4):867–905, Jul 2019.

\bibitem[Xic97]{Xico}
Miguel Xicotencatl.
\newblock {\em Orbit configuration spaces, infinitesimal braid relations in
  homology and equivariant loop spaces}.
\newblock University of Rochester, 1997.
\newblock Ph.D. Thesis.

\bibitem[Xic00]{XicoAnt1}
Miguel~A. Xicot\'{e}ncatl.
\newblock On orbit configuration spaces and the rational cohomology of {$F({\bf
  R}{\rm P}^n,k)$}.
\newblock In {\em Une d\'{e}gustation topologique [{T}opological morsels]:
  homotopy theory in the {S}wiss {A}lps ({A}rolla, 1999)}, volume 265 of {\em
  Contemp. Math.}, pages 233--249. Amer. Math. Soc., Providence, RI, 2000.

\end{thebibliography}

\end{document}